\documentclass[12pt,reqno]{amsart}

\usepackage[left=3.7cm,right=3.7cm]{geometry}
\usepackage{amssymb}
\usepackage{graphicx}
\usepackage{amscd}
\usepackage[pagebackref]{hyperref}
\usepackage{color}
\usepackage{tabularx}
\usepackage[table]{xcolor}
\usepackage{float}
\usepackage{graphics,amsmath,amssymb}
\usepackage{amsthm}
\usepackage{amsfonts}
\usepackage{latexsym}
\usepackage{epsf}
\usepackage{xifthen}
\usepackage{mathrsfs}
\usepackage{dsfont}
\usepackage{makecell}
\usepackage{subfig}
\usepackage{amsmath}
\allowdisplaybreaks[4]
\usepackage{listings}
\usepackage{etoolbox}
\usepackage{fancyhdr}
\usepackage{pdflscape}
\usepackage[title,toc,titletoc]{appendix}
\usepackage{enumitem}
\usepackage[noadjust]{cite}
\usepackage{diagbox}
\usepackage[object=vectorian]{pgfornament} %%  http://altermundus.com/pages/tkz/ornament/index.html
\usepackage{lipsum,tikz}
\usetikzlibrary{cd,fit,matrix,arrows,decorations.pathmorphing}
\tikzset{commutative diagrams/.cd}
\usepackage{accents}
\usepackage{stmaryrd}
\usepackage{mathtools}
\usepackage{etextools}
\usepackage[OT2,T1]{fontenc}

\hypersetup{
	colorlinks=true, %set true if you want colored links
	linktoc=all, %set to all if you want both sections and subsections linked
	linkcolor=blue} %choose some color if you want links to stand out

\numberwithin{equation}{section}

\theoremstyle{theorem}
\newtheorem{theorem}{Theorem}[section]
\newtheorem*{theorem*}{Theorem}

\newtheorem{corollary}[theorem]{Corollary}
\newtheorem{lemma}[theorem]{Lemma}
\newtheorem{proposition}[theorem]{Proposition}

\newtheorem{innercustomgeneric}{\customgenericname}
\providecommand{\customgenericname}{}
\newcommand{\newcustomtheorem}[2]{%
	\newenvironment{#1}[1]
	{%
		\renewcommand\customgenericname{#2}%
		\renewcommand\theinnercustomgeneric{##1}%
		\innercustomgeneric
	}
	{\endinnercustomgeneric}
}
\newcustomtheorem{ctheorem}{Theorem}
\newcustomtheorem{clemma}{Lemma}

\theoremstyle{definition}
\newtheorem{definition}[theorem]{Definition}

\newtheorem*{example*}{Example}
\newtheorem*{examples*}{Examples}
\newtheorem{remark}[theorem]{Remark}
\newtheorem*{remark*}{Remark}
\newtheorem*{remarks*}{Remarks}
\newtheorem*{note*}{Note}

\newtheoremstyle{named}{}{}{\itshape}{}{\bfseries}{.}{.5em}{#1\thmnote{ #3}}
\theoremstyle{named}

\DeclareSymbolFont{cyrletters}{OT2}{wncyr}{m}{n}
\DeclareMathSymbol{\cyZh}{\mathalpha}{cyrletters}{'021}

\newcommand\tl{\widetilde} 
\newcommand\hhat{\widehat} %\hat is too narrow; avoid

\DeclarePairedDelimiter{\abs}{\lvert}{\rvert}

\DeclarePairedDelimiter{\set}{\{}{\}}

\DeclarePairedDelimiter{\parens}{\lparen}{\rparen}

\newcommand\Z{{\mathbb Z}}
\newcommand\N{{\mathbb N}}

\newcommand\C{{\mathbb C}}

\newcommand\R{{\mathbb R}}

\newcommand\Fq{{\mathbb F}_q}

\newcommand\subeq{\subseteq}

\newcommand\supeq{\supseteq}

\newcommand{\qbinom}[2]{{#1\brack #2}}

\newcommand{\zetahat}{\hhat{\zeta}}
\newcommand{\nuhat}{\hhat{\nu}}
\newcommand{\fc}{\mathfrak{c}}
\newcommand{\Fqsq}{\mathbb{F}_{q^2}}

\DeclareMathOperator{\Aut}{Aut}

\DeclareMathOperator{\Spec}{Spec}

\newcommand{\Quot}{\mathrm{Quot}}
\newcommand{\Hilb}{\mathrm{Hilb}}
\newcommand{\Coh}{\mathrm{Coh}}

\newcommand{\type}{\mathrm{type}}

\newcommand{\AG}{\mathbf{AG}}
\newcommand{\Br}{\mathbf{Br}}
\newcommand{\dAG}{\prescript{\dagger\!}{}{\mathbf{AG}}}
\newcommand{\dBr}{\prescript{\dagger}{}{\mathbf{Br}}}
\newcommand{\ddBr}{\prescript{\ddagger}{}{\mathbf{Br}}}

\newcommand{\ut}{\undertilde}

\newcommand{\rk}{\operatorname{rk}}
\newcommand{\TR}{\mathrm{TR}}
\newcommand{\cTR}{\mathrm{cTR}}
\newcommand{\Fl}{\mathrm{Fl}}
\newcommand{\Gr}{\mathrm{Gr}}
\newcommand{\pr}{\mathrm{pr}}

\newcommand{\cX}{\mathcal{X}}
\newcommand{\cV}{\mathcal{V}}

\newcommand{\LHS}{\operatorname{LHS}}

\title[Inert quadratic orders]{Multiple Rogers--Ramanujan type identities for inert quadratic orders}

\author[S. Chern]{Shane Chern}
\address[S. Chern]{Fakult\"at f\"ur Mathematik, Universit\"at Wien, Oskar-Morgenstern-Platz 1, Wien 1090, Austria}
\email{chenxiaohang92@gmail.com, xiaohangc92@univie.ac.at}

\author[Y. Huang]{Yifeng Huang}
\address[Y. Huang]{Department of Mathematics, University of Southern California, 3650 S Vermont Ave, Los Angeles, CA 90089, U.S.A.}
\email{yifenghu@usc.edu}

\date{}



\begin{document}
	
\sloppy

\begin{abstract}
	We compute the Quot and finitized Coh zeta functions of the inert quadratic orders $\mathbb{F}_q[[T]]+T^{m}\mathbb{F}_{q^{2}}[[T]]$ for every $m\geq 1$ in terms of a $2m$-fold multisum, and then show this multisum equals an $m$-fold Bressoud sum. This proves a recent conjecture of the second author, rounding up the line of exploration in the series of work by the authors and Jiang. The equality between the $2m$-fold multisum and the $m$-fold Bressoud sum is built upon generalizing the multisum by introducing a ``ghost'' parameter $a$ to its summands. We then show that such an $a$-generalization is surprisingly $a$-independent by purely $q$-theoretic techniques. Finally, we propose a refined multisum that interpolates two versions of Quot zeta functions for all three types of quadratic orders.
\end{abstract}

\maketitle

\setcounter{tocdepth}{1}
\tableofcontents

\section{Introduction}

A central program in modern algebra, geometry, and combinatorics involves the study of generating functions that enumerate ``objects'' on a variety $X$. These enumerations are encoded by invariants, such as point counts over finite fields or motives in the Grothendieck ring of varieties, associated with various moduli spaces.

Of particular interest are the moduli stack of zero-dimensional coherent sheaves, $\Coh_n(X)$, and the rank-$r$ Quot schemes of points, $\Quot_n(\mathcal{O}_X^{\oplus r})$, which parametrize $0$-dimensional degree-$n$ quotients of $\mathcal{O}_X^{\oplus r}$. These spaces are closely related --- the Hilbert scheme of points, $\Hilb_n(X)$, is the rank-$1$ case, while the stack $\Coh_n(X)$ can be understood as a ``framing-rank-infinity'' limit of the rank-$r$ Quot schemes \cite{huangjiang2023torsionfree}. The study of these spaces is, by definition, the study of the high-rank, degree-zero Donaldson--Thomas theory on $X$; see, for instance, \cite{ricolfi17}. 

The case of planar curve singularities is exceptionally intriguing. The rank-$1$ story, focusing on $\Hilb_n(X)$, is already remarkably rich, featuring the celebrated Oblomkov--Rasmussen--Shende conjecture which connects the geometry of Hilbert schemes to link homology \cite{ors2018homfly}. Recent works by the authors and Jiang \cite{huangjiang2023torsionfree,shane2024multiple,huang2025coh}, in an attempt to ultimately extend this picture to $\Quot_n(\mathcal{O}_X^{\oplus r})$ and $\Coh_n(X)$, revealed new phenomena exclusive to the high-rank setting, namely, the emergence of multiple Rogers--Ramanujan type identities.

The primary target of these investigations is the family of toric singularities $y^a=x^b$, where $a,b\geq 1$. This family is of significant interest in algebraic combinatorics, as its rank-$1$ story is known to be governed by Catalan combinatorics \cite{gorskymazin2013compactified1,gmo2023generic,gmv2016affine,gmv2017rational}. The high-rank study can thus be viewed as an attempt to find a regime for a high-rank generalization of $q,t$-Catalan numbers, one whose rank-to-infinity limit encodes these multiple Rogers--Ramanujan type identities. A work in preparation by the second author, Jiang and Oblomkov finds the corresponding identities for \emph{unibranched} singularities concerning the case where $a$ and $b$ are coprime, thereby confirming the presence of this generalized Catalan structure. However, the \emph{multibranched} case with $a$ and $b$ not coprime appears much more subtle, even for the first non-trivial case $(a,b)=(2,2m)$. The subtlety of general non-coprime pairs $(a,b)$ is already hinted at in rank one, by the occurrence of multiple notions of non-coprime $q,t$-Catalan numbers \cite{gmv2017rational}. Prior study of this case by the second author \cite{huang2025coh} led to a key insight --- for multibranched singularities, a purely geometric perspective over $\C$ is insufficient; instead, an \emph{arithmetic} consideration is essential, requiring the study of all possible models over non-algebraically closed fields that differ by Galois twists.

The goal of the present paper is to leverage this arithmetic perspective to completely settle the $(2,2m)$ case.

We restrict our attention to the arithmetic setting, focusing on point counts over a finite field $\Fq$. We begin by defining the relevant Coh and Quot zeta functions, formulated over a commutative ring $R$.

Let $R$ be a commutative ring with $1$. The \emph{Coh zeta function} of $R$ is defined as the (formal) Dirichlet series
\begin{equation}
	\zetahat_R(s):=\sum_{Q\in \mathbf{FinMod_R}/{\sim}} \frac{1}{\abs{\Aut_R(Q)}} \abs{Q}^{-s},
\end{equation}
summing over all isomorphism classes of finite(-cardinality) modules over $R$. Given an $R$-module $M$, the \emph{Quot zeta function} of $M$ over $R$ is defined as
\begin{equation}
	\zeta_M(s)=\zeta_M^R(s):=\sum_{L\subeq_R M} \abs{M/L}^{-s},
\end{equation}
summing over all finite-index $R$-submodules of $M$. These Dirichlet series are said to be well-defined if for each $n\in \Z_{\geq 1}$, there are only finitely many summands contributing to $n^{-s}$. Specifically, of key concern is the \emph{rank-$r$ Quot zeta function} of $R$, which is simply the Quot zeta function of the free module $R^r$.

These zeta functions are directly connected to point counts of the stack of coherent sheaves and the Quot scheme of points in the sense that if $X = \Spec R$ is an affine variety over $\Fq$, then
\begin{align*}
	\zeta_{R^r}(s)&=\sum_{n\geq 0}\abs{\Quot_n(\mathcal{O}_X^r)(\Fq)} t^n,\\
	\zetahat_{R}(s)&=\sum_{n\geq 0}\abs{\Coh_n(X)(\Fq)} t^n,
\end{align*}
where $t:=q^{-s}$.

The particular normalization of the Quot zeta function we shall consider, which best elucidates the connection to the Coh zeta function, is the \emph{finitized rank-$n$ Coh zeta function}
\begin{equation}
	\zetahat_{R,n}(s):=\zeta^R_{R^n}(s+n).
\end{equation}
According to \cite{huangjiang2023torsionfree}, we have the coefficient-wise limit
$$\lim_{n\to \infty} \zetahat_{R,n}(s)=\zetahat_R(s),$$
justifying the terminology.

The rings we focus on are as follows. Fix a finite field $\Fq$ and an integer $m\geq 1$. Following the notation of \cite{huang2025coh}, define
\begin{align*}
	R_{2,2m+1}&:=\Fq[[T^2,T^{2m+1}]]\subeq \Fq[[T]],\\
	R_{2,2m}&:=\Fq[[X,Y]]/(Y(Y-X^m)),\\
	R_{2,2m}'&:=\Fq[[T]]+T^m \Fqsq[[T]]\subeq \Fqsq[[T]].
\end{align*}
The rings $R_{2,2m+1}$ and $R_{2,2m}$ are the germs of the $(2,2m+1)$ cusp singularity and the $(2,2m)$ node singularity, respectively. The ring $R_{2,2m}'$ is the quadratic twist of $R_{2,2m}$ over the non-algebraically closed field $\Fq$. These three families of rings precisely classify all quadratic orders over the function field $\Fq((X))$ --- $R_{2,2m+1}$ are the \emph{ramified} orders, $R_{2,2m}$ are the \emph{split} orders, and $R_{2,2m}'$ are the \emph{inert} orders.

It has recently been discovered by the second author and Jiang \cite{huangjiang2023torsionfree} and the first author \cite{shane2024multiple} that the finitized Coh zeta functions are closely tied to several classical $q$-hypergeometric series in the ramified and split cases. We first recall the \emph{$t$-deformed finite Andrews--Gordon sum}, which is defined by
\begin{align}\label{eq:AG-finite-t}
	\AG_n^{(2m+3)}(t,q) := (q)_n \sum_{n_1,\ldots,n_m\ge 0} \frac{t^{\sum_{i=1}^m 2n_i} q^{\sum_{i=1}^m n_i^2}}{(q)_{n-n_m}(q)_{n_m-n_{m-1}}\cdots (q)_{n_2-n_1} (q)_{n_1}},
\end{align}
where we have adopted the standard \emph{$q$-Pochhammer symbols} for $N\in\mathbb{N}\cup\{\infty\}$:
\begin{align*}
	(A)_N = (A;q)_N:=\prod_{j=0}^{N-1} (1-A q^j),
\end{align*}
with the following compact notation also used:
\begin{align*}
	(A_1,A_2,\ldots,A_r;q)_N :=(A_1;q)_N(A_2;q)_N\cdots (A_r;q)_N.
\end{align*}
The limiting case of \eqref{eq:AG-finite-t} at $n\to\infty$ for the $t=1$ specialization appears on the sum side of the well-known \emph{``central'' Andrews--Gordon identity}:
\begin{align}\label{eq:AG-infinite}
	\AG_\infty^{(2m+3)}(1,q) = \frac{(q^{m+1},q^{m+2},q^{2m+3};q^{2m+3})_\infty}{(q;q)_\infty},
\end{align}
which was discovered by Andrews \cite[p.~4082, Theorem~1]{And1974} in relation to Gordon's partition theorem \cite{Gor1961}. It is notable that identities in the form of \eqref{eq:AG-infinite} are usually called \emph{multiple Rogers--Ramanujan type identities} due to their connection to the classical work of Rogers \cite{Rog1894} and Ramanujan \cite{Ram1914}. Next, our second $q$-hypergeometric series of interest is the \emph{$t$-deformed finite Bressoud sum} defined by
\begin{align}\label{eq:Br-finite-t}
	\Br_n^{(2m+2)}(t,q) := (q)_n \sum_{n_1,\ldots,n_m\ge 0} \frac{t^{\sum_{i=1}^m 2n_i} q^{\sum_{i=1}^m n_i^2}}{(q)_{n-n_m}(q)_{n_m-n_{m-1}}\cdots (q)_{n_2-n_1} (q)_{n_1}(-tq)_{n_1}}.
\end{align}
Similarly, setting $t=1$ and taking the $n\to \infty$ limit, we witness another multiple Rogers--Ramanujan type identity known as the \emph{``central'' Bressoud identity} \cite[p.~385, Theorem~1]{Bre80}:
\begin{align}
	\Br_\infty^{(2m+2)}(1,q) = \frac{(q^{m+1},q^{m+1},q^{2m+2};q^{2m+2})_\infty}{(q;q)_\infty}.
\end{align}

The $t$-deformations we have chosen may not appear to be the most natural from the $q$-theoretic point of view; see Section~\ref{sec:t-deform} for further discussions. However, these $t$-deformed $q$-hypergeometric sums, surprisingly, are the ones that give the finitized Coh zeta functions.

\begin{theorem}[cf.~\cite{huangjiang2023torsionfree} and \cite{shane2024multiple}]
	For each $m\ge 1$, the finitized Coh zeta functions for the ramified and split orders are given by
	\begin{align}
		\zetahat_{R_{2,2m+1},n}(s) &= \frac{1}{(tq^{-1};q^{-1})_n} \AG_n^{(2m+3)}(t,q^{-1}),\label{eq:coh-zeta-ramified}\\
		\zetahat_{R_{2,2m},n}(s) &= \frac{1}{(tq^{-1};q^{-1})_n} \Br_n^{(2m+2)}(-t,q^{-1}),\label{eq:coh-zeta-split}
	\end{align}
	where $t:=q^{-s}$.
\end{theorem}

A puzzling feature of \eqref{eq:coh-zeta-split} is the presence of $-t$.
In the direction toward a geometric object that more directly corresponds to the exactly $t$-deformed Bressoud sum $\Br_n^{(2m+2)}(t,q)$, the second author \cite{huang2025coh} shifted his attention to inert quadratic orders, objects that interplay not only with geometry but also with arithmetic. In particular, a similar sum-like expression, aligning with those in \eqref{eq:coh-zeta-ramified} and \eqref{eq:coh-zeta-split}, was conjectured for the finitized Coh zeta function in this inert case; see \cite[Conjecture~1.2]{huang2025coh}. The main objective of this work is to settle this puzzle, thereby completing the full picture of quadratic orders.

\begin{theorem}[{\cite[Conjecture~1.2]{huang2025coh}}]\label{thm:main}
	For each $m\ge 1$, the finitized Coh zeta function for the inert quadratic order $R'_{2,2m}$ is given by the directly $t$-deformed finite Bressoud sum
	\begin{align}\label{eq:main}
		\zetahat_{R'_{2,2m},n}(s) = \frac{1}{(tq^{-1};q^{-1})_n} \Br_n^{(2m+2)}(t,q^{-1}).
	\end{align}
\end{theorem}

As we will see in Corollary~\ref{coro:zeta-new}, the finitized Coh zeta function $\zetahat_{R'_{2,2m},n}(s)$ is closely related to the following $2m$-fold multisum:
\begin{align*}
	&\sum_{\substack{r_1,\ldots,r_m\ge 0\\s_1,\ldots,s_m\ge 0}} \frac{t^{\sum_{i=1}^m (2r_i-s_i)} q^{\sum_{i=1}^m (r_i^2-r_is_i+s_i^2)}(-q;q)_{r_1}}{(q;q)_{n-r_m} (q;q)_{r_{m}-r_{m-1}}\cdots (q;q)_{r_2-r_1}(q;q)_{r_1}(t^2q^2;q^2)_{r_1}(-q;q)_{s_1}}\notag\\
	&\times \qbinom{r_m-s_{m-1}}{r_m-s_m}_q\qbinom{r_{m-1}-s_{m-2}}{r_{m-1}-s_{m-1}}_q\cdots \qbinom{r_{2}-s_1}{r_{2}-s_{2}}_q\qbinom{r_1}{r_1-s_1}_q,
\end{align*}
where the \emph{$q$-binomial coefficients} are defined by
\begin{align*}
	\qbinom{N}{M} = \qbinom{N}{M}_q:=\begin{cases}
		\dfrac{(q;q)_N}{(q;q)_M(q;q)_{N-M}}, & \text{if $0\le M\le N$},\\[10pt]
		0, & \text{otherwise}.
	\end{cases}
\end{align*}
In the meantime, a similar multisum appears in the split case \cite[Theorem~9.11]{huangjiang2023torsionfree}:
\begin{align*}
	&\sum_{\substack{r_1,\ldots,r_m\ge 0\\s_1,\ldots,s_m\ge 0}} \frac{t^{\sum_{i=1}^m (2r_i-s_i)} q^{\sum_{i=1}^m (r_i^2-r_is_i+s_i^2)}}{(q;q)_{n-r_m} (q;q)_{r_{m}-r_{m-1}}\cdots (q;q)_{r_2-r_1}(tq;q)_{r_1}^2(q;q)_{s_1}}\notag\\
	&\times \qbinom{r_m-s_{m-1}}{r_m-s_m}_q\qbinom{r_{m-1}-s_{m-2}}{r_{m-1}-s_{m-1}}_q\cdots \qbinom{r_{2}-s_1}{r_{2}-s_{2}}_q\qbinom{r_1}{r_1-s_1}_q.
\end{align*}
Observing the resemblance between the two expressions, it becomes natural to consider the following generalized multisum, with an additional parameter introduced:
\begin{align}\label{eq:X-def}
	&\cX_N^{(m)}(a,t,q)\notag\\
	&\quad:=(atq)_N\sum_{\substack{r_1,\ldots,r_m\ge 0\\s_1,\ldots,s_m\ge 0}} \frac{a^{\sum_{i=1}^m s_i}t^{\sum_{i=1}^m (2r_i-s_i)} q^{\sum_{i=1}^m (r_i^2-r_is_i+s_i^2)} (aq)_{r_1}}{(q)_{N-r_m} (q)_{r_{m}-r_{m-1}}\cdots (q)_{r_2-r_1}(q)_{r_1}(tq)_{r_1}(atq)_{r_1}(aq)_{s_1}}\notag\\
	&\quad\ \quad\times \qbinom{r_m-s_{m-1}}{r_m-s_m}\qbinom{r_{m-1}-s_{m-2}}{r_{m-1}-s_{m-1}}\cdots \qbinom{r_{2}-s_1}{r_{2}-s_{2}}\qbinom{r_1}{r_1-s_1}.
\end{align}
A crucial result that we are about to establish is that this newly inserted parameter $a$ is indeed a \emph{ghost}:

\begin{theorem}\label{th:X-a-indep}
	For every $N\ge 0$, the $2m$-fold sum $\cX_N^{(m)}(a,t,q)$ is independent of $a$. In particular,
	\begin{align}\label{eq:X-nice-exp}
		\cX_N^{(m)}(a,t,q) = \sum_{n_1,\ldots,n_m\ge 0} \frac{t^{\sum_{i=1}^m 2n_i} q^{\sum_{i=1}^m n_i^2}}{(q)_{N-n_m} (q)_{n_{m}-n_{m-1}}\cdots (q)_{n_2-n_1}(q)_{n_1}(tq)_{n_1}}.
	\end{align}
\end{theorem}

From a $q$-perspective, the $a$-independence for $\cX_N^{(m)}(a,t,q)$ is \emph{unexpected}, somehow even \emph{counter-intuitive}, because the parameter $a$ contributes to the multisum with respect to \emph{every} index $s_1,\ldots,s_m$. However, we shall show that this parameter eventually vanishes after a sequence of $q$-hypergeometric transformations.

So far, the arithmetic and geometric considerations have generated \emph{predictions} about $q$-series. Can this go the other way, where purely $q$-theoretic considerations generate guesses about algebraic geometry? The answer is optimistic. The study of the Quot zeta function has revealed phenomena that appear to have a structural explanation, but no framework has yet been proposed. In Section~\ref{sec:further-interpolations}, purely by interpolating the known Quot zeta functions, we come up with a $q$-hypergeometric definition of a refined polynomial $\cyZh_{R,n}(u,t,z)$ where $R$ is a quadratic order. We verify that certain uniform $q$-series identities persist on this refinement (cf.~Theorems \ref{thm:master-reflection} and \ref{thm:Zh-at-root}), suggesting that $\cyZh_{R,n}(u,t,z)$ might underlie an arithmetic or geometric framework that potentially extends beyond the case of quadratic orders.

\subsection*{Plan of the paper}

In Section~\ref{sec:prelim}, we summarize the geometric preliminaries in \cite{huangjiang2023torsionfree} and set up the necessary notation. In Section~\ref{sec:computation}, we apply the framework introduced earlier to compute the Quot and finitized Coh zeta functions for the inert quadratic orders, reaching a $2m$-fold multisum in Corollary~\ref{coro:zeta-new}. This $2m$-fold expression results in a proof of Theorem~\ref{thm:main} after assuming Theorem~\ref{th:X-a-indep}. In Section~\ref{sec:a-indep}, we complete the proof of Theorem~\ref{th:X-a-indep} by a purely $q$-theoretic approach. Finally, we raise two closing discussions, one on the selection of deformations for multiple Rogers--Ramanujan type sums in Section~\ref{sec:t-deform}, and the other on further interpolations of the normalized Quot zeta functions for the three quadratic orders in Section~\ref{sec:further-interpolations}.

\section{Preliminaries}\label{sec:prelim}

\subsection{Lattice zeta functions}

We recall some relevant definitions and properties from \cite[Sections~4 and 5]{huangjiang2023torsionfree}. Let $R$ be an arithmetic local order, $\tl R$ its normalization, and $K$ its total ring of fractions. The normalization $\tl R$ is a product of rings $\tl R=\prod_{i=1}^{b(R)} \tl R_i$, where for each $i$ with $1\leq i\leq b(R)$, the ring $\tl R_i$ is a discrete valuation ring (DVR), with residue field denoted by $\kappa_i$. Denote $q_i:=\abs{\kappa_i}$. The \emph{conductor} of $R$, denoted $\fc_R$, is the largest ideal of $\tl R$ that is contained in $R$. We have the finitude of $\abs{\tl R/R}$ and $\abs{R/\fc_R}$.

\begin{remark}
	In this paper, we will only work with the example
	\begin{align*}
		R=R_{2,2m}'=\Fq[[T]]+T^m\Fqsq[[T]].
	\end{align*}
	In this case, $\tl R=\Fqsq[[T]]$ and $K=\Fqsq((T))$ with $b(R)=1$. Also, we have $\kappa_1=\Fqsq$ with $q_1=q^2$. Finally, $\fc_R=(T^m)\tl R$, and $\abs{\tl R/R}=\abs{R/\fc_R}=q^m$. 
\end{remark}

For any finitely generated $\tl R$-module $\tl Q$, its \emph{rank} along the $i$-th branch is defined as
\begin{equation*}
	\rk_i(\tl Q):=\dim_{\kappa_i} \tl Q\otimes_{\tl R} \kappa_i.
\end{equation*}
An \emph{$R$-lattice} in $K^n$ is a finitely generated $R$-submodule $L\subeq K^n$ that spans $K^n$ over $K$. An $R$-\emph{sublattice} of an $R$-lattice $M$ in $K^n$ is an $R$-lattice in $K^n$ that is contained in $M$, or equivalently, a finite-index $R$-submodule of $M$. We use the notation $L\subeq_R M$ to mean that $L$ is an $R$-sublattice of $M$.

Given $L_1\subeq_R L_2$, define $(L_2:L_1):=\abs{L_2/L_1}$. Given an $R$-lattice $L$ in $K^n$, the \emph{extension} (or \emph{saturation}) of $L$ is defined as $\tl RL$, the smallest $\tl R$-lattice in $K^n$ containing $L$. Recall the Quot zeta function
\begin{equation*}
	\zeta_M(s)=\zeta_M^R(s):=\sum_{L\subeq_R M} (M:L)^{-s},
\end{equation*}
where the sum ranges over all finite-index $R$-submodules of $M$. If $M$ is an $R$-lattice, then the sum equivalently ranges over $R$-sublattices of $M$. If $R=\tl R$ is a DVR (product), the Quot zeta function of any $R$-lattice can be computed by the Solomon formula \cite{solomon1977zeta}:
\begin{equation}\label{eq:solomon-formula}
	\zeta_{\tl R^n}^{\tl R}(s)=\prod_{i=1}^l (q_i^{-s};q_i)_n^{-1}.
\end{equation}
If $R\neq \tl R$, then we will need the following formula to compute the Quot zeta functions of lattices:

\begin{theorem}
	[cf.~{\cite[Theorem~5.11]{huangjiang2023torsionfree}}]
	Let $R$ be an arithmetic local order, and $M$ be an $R$-lattice in $K^n$. Let $\ut M$ be any $\tl R$-lattice contained in $M$. The normalized $\zeta_M^R(s)$ is given by the finite sum
	\begin{equation}\label{eq:rational-formula-arithmetic}
		\nu_M^R(s) :=\frac{\zeta_M^R(s)}{\zeta_{\tl R^n}^{\tl R}(s)} = \sum_{(\tl L,L)\in B_R(M;\ut M)} (M:L)^{-s} \prod_{i=1}^{b(R)} (q_i^{-s};q_i)_{\rk_i(\tl L/\ut M)},
	\end{equation}
	where
	\begin{equation}
		B_R(M;\ut M)=\big\{(\tl L,L)\;:\;\tl L\supeq \ut M,\, L\subeq_R M,\, \tl RL=\tl L\big\}.
	\end{equation}
	\label{thm:rationality-arithmetic}
\end{theorem}

\subsection{Combinatorics of DVR modules}

Let $(V,\pi,\Fq)$ be a DVR with uniformizer $\pi$ and finite residue field $\Fq$. Given a finite module $M$ over $V$, the \emph{type} of $M$ over $V$ is the unique integer partition $\lambda=(\lambda_1, \lambda_2, \ldots, \lambda_\ell)$ such that $M\simeq_V M_V(\lambda):=\bigoplus_i V/\pi^{\lambda_i}$. Here an \emph{integer partition} $\lambda$ is a weakly decreasing sequence of positive integers $\lambda_1\ge \lambda_2\ge \cdots \ge\lambda_\ell$ with its \emph{size} defined by $|\lambda| := \sum_i \lambda_i$; see \cite{And1998} for a rich introduction. The \emph{rank} of a $V$-module $M$ is given by
\begin{align*}
	\rk_V(M):=\dim_{\Fq} W\otimes_V \Fq.
\end{align*}
The rank and type are related by
\begin{align*}
	\rk_V(M_V(\lambda))=\lambda_1' =: \ell,
\end{align*}
the \emph{length} of the partition $\lambda$.

Let $B=V/\pi^m$ for some $m\geq 1$. Then a $B$-module $M$ is simply a $V$-module whose type $\lambda$ satisfies $\lambda_1\leq m$. The category of $B$-modules is a full subcategory of the category of $V$-modules, so for $B$-modules, the notions $\simeq_B$ and $\simeq_V$ are equivalent, and $\subeq_B$ and $\subeq_V$ mean the same thing. For a finite $B$-module $M$, we refer to the type of $M$ over $V$ as the type of $M$ over $B$, and denote $\type_B(M):=\type_V(M)$. Conversely, for $\lambda_1\leq m$, denote $M_B(\lambda):=M_V(\lambda)$. For $V$-modules $L\subeq M$, the \emph{cotype} of $L$ in $M$ is the type of $M/L$. 

Now let us recall the notion of \emph{Hall polynomial}:

\begin{theorem}[cf.~{\cite[\S II]{macdonaldsymmetric}}]
	For any partitions $\lambda$, $\mu$, and $\nu$, there is a universal polynomial (called the \emph{Hall polynomial}) $g^\lambda_{\mu\nu}(t)\in \Z[t]$ such that for any DVR $(V,\pi,\Fq)$, the number of $V$-submodules of $M_V(\lambda)$ of type $\mu$ and cotype $\nu$ is given by $g^{\lambda}_{\mu\nu}(q)$. Moreoever, $g^{\lambda}_{\mu\nu}(t)=g^\lambda_{\nu\mu}(t)$.
\end{theorem}

We will only use a \emph{coarser version} of the Hall polynomial:

\begin{definition}
	For partitions $\lambda$ and $\mu$, let $g^\lambda_\mu(t):=\sum_\nu g^\lambda_{\mu\nu}(t)$. 
\end{definition} 

Clearly, $g^\lambda_\mu(q)$ counts submodules of $M_V(\lambda)$ of type $\mu$, as well as submodules of $M_V(\lambda)$ of cotype $\mu$. An explicit expression of $g^\lambda_\mu(t)$ is as follows:

\begin{theorem}[cf.~{\cite{warnaar2013remarks}}]\label{thm:hall_skew}
	For any partitions $\lambda$ and $\mu$, we have
	\begin{equation}\label{eq:g-formula}
		g^{\lambda}_{\mu}(t)=t^{\sum_{i\geq 1}\mu_i'(\lambda_i'-\mu_i')}\prod_{i\geq 1} {\lambda_i'-\mu_{i+1}' \brack \lambda_i'-\mu_i'}_{t^{-1}},
	\end{equation}
	where as usual, $\lambda_i'$ denotes the $i$-th part of the conjugate partition $\lambda'$ of $\lambda$.
\end{theorem}

We also need two well-known facts about rectangular partitions. 
\begin{lemma}
	Let $(m^n)$ denote the partition $(m,\dots,m)$ with $n$ parts of $m$. Then for any $\mu\subeq (m^n)$, there is a unique partition $\nu$ such that $g^{(m^n)}_{\mu\nu}\neq 0$. Moreover, $\nu$ is the complement of $\mu$ in $(m^n)$ (so $\nu_i=m-\mu_{n+1-i}$ for $1\leq i\leq n$ and $\nu_{n+1}=0$), denoted by $(m^n)-\mu$. 
	\label{lem:rectangular}
\end{lemma}

\begin{lemma}\label{lem:transitive}
	Let $M$ be a $V$-module of type $(m^n)$, and let $N$ and $N'$ be two $V$-submodules of the same type $\mu$. Then there exists $g\in \Aut_V(M)$ such that $gN=N'$.
\end{lemma}
\begin{proof}
	Note that $M$ is a free module of rank $n$ over the ring $B=V/(\pi^m)$. Consider a $V$-module $Q$ of type $(m^n)-\mu$, which is the cotype of both $N$ and $N'$ by Lemma~\ref{lem:rectangular}. Thus, there are surjections $\varphi$ and $\varphi'$ from $M$ to $Q$ whose kernels are $N$ and $N'$, respectively. The conclusion then follows from applying \cite[Lemma~5.1]{huangjiang2023torsionfree} to the two surjections.
\end{proof}

Finally, we recall the number of automorphisms of a DVR module:

\begin{lemma}[cf.~{\cite[p.~181]{macdonaldsymmetric}}]
	For any partition $\lambda$, we have
	\begin{equation}\label{eq:aut}
		\abs{\Aut_V(M_V(\lambda))}=a_\lambda(q):=q^{\sum \lambda_i'^2} \prod_{i\geq 1}(q^{-1};q^{-1})_{\lambda_i'-\lambda_{i+1}'}.
	\end{equation}
\end{lemma}

\section{Computation of the Quot zeta function}\label{sec:computation}

Throughout this section, let $k=\Fq$ and $l=\Fqsq$. For $m\geq 1$, let
\begin{align*}
	R=R_{2,2m}'=k[[T]]+T^m l[[T]].
\end{align*}
Also, write $M=R^n$. The goal of this section is to give the first explicit formula of $\zeta_M^R(s)$. Recall $\tl R=l[[T]]$ and $K=l((T))$. In addition, $\kappa_1=l$ and $\fc=\fc_R=(T^m)l[[T]]$. We will apply Theorem~\ref{thm:rationality-arithmetic} with the choice $\ut M=\fc M=(T^m) l[[T]]^n$. Define rings
\begin{equation}
	A:=\frac{R}{\fc}\simeq \frac{k[[T]]}{T^m} \subeq \tl A:=\frac{\tl R}{\fc}\simeq \frac{l[[T]]}{T^m}.
\end{equation}

\subsection{Boundary lattices}

We begin with the part of arguments that are in common with \cite[Section~9]{huangjiang2023torsionfree}. For an $\tl R$-lattice $\tl L$, let its \emph{extension fiber} be 
\begin{equation*}
	E_R(\tl L):=\big\{L\subeq_R \tl L \;:\; \tl R L=\tl L\big\}.
\end{equation*}
If furthermore $L$ is an $R$-lattice, then the \emph{restricted extension fiber} is
\begin{equation*}
	E_R(\tl L;L):=\big\{L\subeq_R L \;:\; \tl R L=\tl L\big\}.
\end{equation*}
The set $B_R(M;\ut M)$ used in Theorem~\ref{thm:rationality-arithmetic} is naturally in bijection with the set
\begin{equation}\label{eq:boundary-locus}
	\big\{(L_b,L) \;:\; \text{$\ut M\subeq L_b\subeq_R M$ with $\tl RL_b\cap M=L_b$, and $L\in E_R(\tl R L_b;L_b)$}\big\}
\end{equation}
by the rules $L_b=\tl L\cap M$ and $\tl L=\tl RL_b$; see \cite[Section 9]{huangjiang2023torsionfree} for details. By abuse of notation, we denote the set \eqref{eq:boundary-locus} by $B_R(M;\ut M)$ as well. Let $\partial_R(M;\ut M)$ be the set consisting of \emph{boundary lattices} $L_b$, namely, those with $\ut M\subeq L_b\subeq_R M$ and $\tl RL_b\cap M=L_b$. Then the forgetful map $B_R(M;\ut M)\to \partial_R(M;\ut M)$ is surjective whose fiber over $L_b$ is $E_R(\tl RL_b;L_b)$.

%\begin{lemma}[{cf.~{\cite[Lemma~9.9]{huangjiang2023torsionfree}}}]\label{lem:boundary-lattice-invariants}
\begin{lemma}\label{lem:boundary-lattice-invariants}
	Assume the notation above. Let $L_b$ be any rank-$n$ $R$-lattice such that $\ut M \subeq L_b \subeq M$, and let $\lambda$ be the type of the $A$-module $M/L_b$. Then
	\begin{enumerate}[label={\textup{(\alph*)}},leftmargin=*,labelsep=0cm,align=left,itemsep=2pt]
		\item We have $\tl RL_b\cap M=L_b$, so that $L_b\in \partial_R(M;\ut M)$;
		
		\item The type of the $A$-module $\tl R L_b / L_b$ is $(m^n)-\lambda$;
		
		\item The type of the $\tl A$-module $\tl R L_b / \ut M$ is $(m^n)-\lambda$.
		In particular, the rank of $\tl R L_b/\ut M$ along the first (and only) branch of $\tl R$ is
		\begin{equation*}
			\rk_1(\tl RL_b/\ut M)=n-\lambda_m'.
		\end{equation*}
	\end{enumerate}
\end{lemma}
\begin{proof}
	Let $\tl L=\tl RL_b$. Consider the tower of lattices
	\[
	\begin{tikzcd}[row sep=tiny, column sep=large]
		& \tl R M \ar[dl, no head] \ar[dr, no head] & \\
		\tl L \ar[dr, no head] & & M \ar[dl, no head] \\
		& L_b \ar[d, no head] & \\
		& \ut M &
	\end{tikzcd}.
	\]
	Define $W_b:=L_b/\ut M$. Modulo $\ut M$, the tower becomes
	\[
	\begin{tikzcd}[row sep=tiny, column sep=large]
		& \tl A^n \ar[dl, no head] \ar[dr, no head] & \\
		\tl A W_b \ar[dr, no head] & & A^n \ar[dl, no head] \\
		& W_b \ar[d, no head] & \\
		& 0 &
	\end{tikzcd}.
	\]
	The $A$-cotype of $W_b$ in $A^n$ is $\lambda$, so the $A$-type of $W_b$ is $(m^n)-\lambda$.
	
	Let $\Theta$ be any element of $l\setminus k$. At this point, a crucial structural observation is that $\tl A=A\oplus \Theta A$; moreover, the multiplication map $A\to \tl A$ by $\Theta$ is injective. Since $W_b\subeq A^n$, this means $\tl AW_b=W_b\oplus \Theta W_b$ (with $\Theta$ acting injectively), so $\tl AW_b\cap A^n=W_b$ and $\tl AW_b/W_b\simeq \Theta W_b\simeq W_b$ as $A$-modules. This proves parts~(a) and (b) of the lemma.
	
	For part~(c), the assertion $\tl AW_b=W_b\oplus \Theta W_b$ with $\Theta$ acting injectively implies that $\tl AW_b\simeq_{\tl A} W_b\otimes_A {\tl A}$, namely, the structure of $\tl AW_b$ is determined intrinsically by the structure of $W_b$. Since for any $1\leq \ell\leq m$,
	\begin{equation}
		\frac{k[[T]]}{(T^\ell)} \otimes_A \tl A \simeq \frac{l[[T]]}{T^\ell},
	\end{equation}
	we conclude the final part of the lemma.
\end{proof}

\subsection{The poset structure of the extension fiber}

Let $\tl L$ be an $\tl R$-lattice, and let $L_b\in E_R(\tl L)$. In order to continue from the previous discussion, we need to understand the set $E_R(\tl L;L_b)$ in terms of the cotype of $L_b$ in $\tl L$. In other words, we need to understand the \emph{lower intervals} of the poset $E_R(\tl L)$ ordered by inclusion. Since every element of $E_R(\tl L)$ contains $\fc \tl L$, the poset $E_R(\tl L)$ is isomorphic to
\begin{equation*}
	E_A(\tl V):=\big\{W\subeq_A \tl V \;:\; \tl AW=\tl V\big\},
\end{equation*}
where $\tl V=\tl L/\fc \tl L \simeq \tl A^n$.

From this point, we diverge from the arguments in \cite{huangjiang2023torsionfree}. In the case of ramified and split orders, the condition $\tl AW=\tl V$ is equivalent to $\tl AW+V_1=V$ for some fixed submodule $V_1\subeq V$, and thus $E_A(\tl V)$ can be parametrized using the second isomorphism theorem; see \cite[Lemma~5.2]{huangjiang2023torsionfree}. In the inert case here, this argument no longer works. Instead, we need a new structural observation.

\subsection{Totally real submodules}

This subsection is not directly needed in the rest of the proof; instead, its \emph{dual} version will be used, and this is to be outlined in the next subsection. However, the moduli spaces defined in this subsection are more intuitive, so we give a detailed discussion to motivate our main idea.

Let $l/k$ be any quadratic field extension. Let $\Theta$ be any element of $l\setminus k$. Let
\begin{align*}
	A=k[[T]]/T^m\subeq \tl A=l[[T]]/T^m.
\end{align*}

\begin{definition}
	Let $\tl V$ be a finite-dimensional $l$-vector space (or a finite-length $\tl A$-module). A $k$-subspace $W\subeq_k \tl V$ (or an $A$-submodule $W\subeq_A \tl V$) is \emph{totally real} if $W\cap \Theta W=0$. 
\end{definition}
Note that for any $W\subeq_k \tl V$ (or $W\subeq_A \tl V$), the intersection $W\cap \Theta W$ is the largest $l$-subspace (or $\tl A$-submodule) contained in $W$; see \cite[Lemma~5.2]{huang2025coh}.

\begin{remark}\label{rmk:sub-real}
	If $\tl V$ is a finite-length $\tl A$-module, $W$ is a totally real $A$-submodule of $\tl V$, and $W'\subeq_A W$ is any $A$-submodule, then clearly $W'$ is also totally real. 
\end{remark}

If $\tl V$ is a finite-length $\tl A$-module and $W$ is a totally real $A$-submodule of $\tl V$, then
\begin{equation*}
	\tl AW=W\oplus \Theta W=lW
\end{equation*}
is isomorphic to $W\otimes_A \tl A$ as an $\tl A$-module; see the proof of Lemma~\ref{lem:boundary-lattice-invariants}. As a result, the $\tl A$-type of $\tl AW$ is the same as the $A$-type of $W$. If $W$ is a totally real $A$-submodule of $\tl V$, then the following relations are equivalent:
\begin{enumerate}[label={\textup{(\arabic*)~}},leftmargin=*,labelsep=0cm,align=left,itemsep=2pt]
	\item $\tl AW=\tl V$;
	
	\item $\dim_k W=\dim_l V$;
	
	\item $\type_A(W)=\type_{\tl A} V$.
\end{enumerate}
In this case, we say $W$ is a \emph{real structure} of $\tl V$.

Let $\tl V$ be a finite-length $l[[T]]$-module. For a partition $\lambda$, define the \emph{totally real Grassmannian}
\begin{equation*}
	\Gr^{\TR}_A(\lambda;\tl V) := \big\{W\subeq_A \tl V \;:\; \text{$W$ totally real and $\type_A(W)=\lambda$}\big\}.
\end{equation*}

\begin{lemma}\label{lem:real-struct-count}
	Let $\tl V$ be an $\tl A$-module of type $\lambda$. Then $\Aut_{\tl A}(\tl V)$ acts transitively on $\Gr^{\TR}_A(\lambda;\tl V)$, the set of real structures of $\tl V$. Moreover, the number of real structures of $\tl V$ is
	\begin{equation}
		\abs{\Gr^{\TR}_A(\lambda;\tl V)}=\frac{a_\lambda(q^2)}{a_\lambda(q)}.
	\end{equation}
\end{lemma}
\begin{proof}
	Using an isomorphism $\tl V\simeq M_{\tl A}(\lambda)$ and the standard real structure $M_A(\lambda)\subeq M_{\tl A}(\lambda)$, we see that $\Gr^{\TR}_A(\lambda;\tl V)$ is nonempty. We fix a real structure $V\subeq \tl V$, so we may identify $\tl V$ with $V\otimes_A \tl A$.
	
	Consider the action of $G=\Aut_{\tl A}(\tl V)$ on $\Gr^\TR_A(\lambda;\tl V)$. It is transitive --- Given any $W\in \Gr^\TR_A(\lambda;\tl V)$, there is an $A$-module isomorphism $W\simeq_A M_A(\lambda)\simeq V$. Tensoring this isomorphism with $\tl A$ gives an automorphism $\tl V=\tl AW\to \tl AV=\tl V$ that sends $W$ to $V$. This proves transitivity.
	
	The stabilizer of $V$ under $G$-action is clearly $H=\Aut_A(V)$. The claim then follows from the orbit-stabilizer theorem.
\end{proof}

For partitions $\lambda$ and $\mu$, define the \emph{totally real flag variety}
\begin{align*}
	\Fl^{\TR}_A(\mu, \lambda;\tl V) := \big\{(W_1,W_2) \;: &\;\;\text{$W_1\subeq_A W_2\subeq_A \tl V$ with $W_1,W_2$ totally real,}\\
	&\;\;\text{and $\type_A(W_1)=\mu$, $\type_A(W_2)=\lambda$}\big\}.
\end{align*}
We have forgetful maps
\[\begin{tikzcd}[sep=scriptsize]
	& \Fl^{\TR}_A(\mu, \lambda;\tl V) \\
	\Gr^{\TR}_A(\mu;\tl V) && \Gr^{\TR}_A(\lambda;\tl V)
	\arrow["{\pr_1}"', from=1-2, to=2-1]
	\arrow["{\pr_2}", from=1-2, to=2-3]
\end{tikzcd}.\]

It turns out that the structures of $\Gr_A^\TR(\lambda;\tl V)$ and $\Gr_A^\TR(\mu;\tl V)$, along with the fibers of $\pr_2$, are simple to understand. The fibers of $\pr_1$ are the objects we will be interested in. Although a direct analysis of these fibers appears difficult, we are able to obtain their point counts when $\tl V$ is of rectangular type, thanks to a transitive group action. The content of this paragraph is made precise in a counting setting, as follows:

\begin{lemma}\label{lem:tot-real}
	Assume the notation above, with $k=\Fq$, $l=\Fqsq$, and $\tl V=\tl A^n$. For partitions $\mu,\lambda\subeq (m^n)$,
	\begin{enumerate}[label={\textup{(\alph*)}},leftmargin=*,labelsep=0cm,align=left,itemsep=2pt]
		\item $\abs{\Gr^\TR_A(\lambda;\tl V)}=g^{(m^n)}_\lambda(q^2) \dfrac{a_\lambda(q^2)}{a_\lambda(q)}$;
		
		\item Every fiber of $\pr_2$ has cardinality $g^\lambda_\mu(q)$;
		
		\item Every fiber of $\pr_1$ has the same cardinality.
	\end{enumerate}
\end{lemma}

\begin{proof}
	\hspace*{0pt}
	\begin{enumerate}[label={\textup{(\alph*)}},leftmargin=*,labelsep=0cm,align=left,itemsep=2pt]
		\item To pick $W\in \Gr^\TR_A(\lambda;\tl V)$, we first choose $\tl W:=\tl AW\simeq \tl A\otimes_A W$, which is an $\tl A$-submodule of $\tl V$ of type $\lambda$. Since the residue field of $\tl A$ is $\Fqsq$, there are $g^{(m^n)}_\lambda(q^2)$ choices for $\tl W$. After choosing $\tl W$, we need to pick $W\in \Gr^\TR_A(\lambda;\tl W)$, which has $a_\lambda(q^2)/a_\lambda(q)$ possibilities by Lemma~\ref{lem:real-struct-count}. This finishes the proof of part~(a).
		
		\item This claim simply follows from Lemma~\ref{rmk:sub-real}.
		
		\item We note that $\pr_1$ is equivariant under the natural action of $G=\Aut_{\tl A}(\tl V)$. 
		By Lemma~\ref{lem:real-struct-count} and Lemma~\ref{lem:transitive}, $G$ acts transitively on the target of $\pr_1$, namely, $\Gr^\TR_A(\mu;\tl V)$. It follows that any two fibers of $\pr_1$ are related by a bijection induced by some $g\in G$. \qedhere
	\end{enumerate}
\end{proof}

\begin{theorem}
	Let $k=\Fq$, $l=\Fqsq$, and $\tl V$ be an $l[[T]]$-module of type $(m^n)$. Then for partitions $\mu,\lambda\subeq (m^n)$, every totally real $k[[T]]$-submodules of $\tl V$ of type $\mu$ is contained in $B(m,n,\lambda,\mu,q)$ totally real $k[[T]]$-submodules of $\tl V$ of type $\lambda$, where
	\begin{equation}
		B(m,n,\lambda,\mu,q)=\frac{g^{(m^n)}_\lambda(q^2) g^\lambda_\mu(q) a_\lambda(q^2)/a_\lambda(q)}{g^{(m^n)}_\mu(q^2) a_\mu(q^2)/a_\mu(q)}.
	\end{equation}
\end{theorem}
\begin{proof}
	What we need is the fiber cardinality of $\pr_1$. By Lemma~\ref{lem:tot-real}(c), this is
	\begin{align*}
		\frac{\abs{\Fl^\TR_A(\mu,\lambda;\tl V)}}{\abs{\Gr^\TR_A(\mu;\tl V)}} \overset{\text{Lem~\ref{lem:tot-real}(b)}}{=} \frac{\abs{\Gr^\TR_A(\lambda;\tl V)}g^\lambda_\mu(q)}{\abs{\Gr^\TR_A(\mu;\tl V)}} \overset{\text{Lem~\ref{lem:tot-real}(a)}}{=} \frac{g^{(m^n)}_\lambda(q^2) g^\lambda_\mu(q) a_\lambda(q^2)/a_\lambda(q)}{g^{(m^n)}_\mu(q^2) a_\mu(q^2)/a_\mu(q)},
	\end{align*}
	as desired.
\end{proof}

\begin{remark}
	We have an elementary simplification
	\begin{align}
		&B(m,n,\lambda,\mu,q)\notag\\
		&\quad=q^{\sum_i (2n-\lambda_i')(\lambda_i'-\mu_i')}(-q^{-1};q^{-1})_{\lambda_1'-\mu_1'}\qbinom{n-\mu_1'}{n-\lambda_1'}_{q^{-2}}\prod_{i\geq 1}\qbinom{\lambda_i'-\mu_{i+1}'}{\lambda_i'-\lambda_{i+1}'}_{q^{-1}},
	\end{align}
	which shows that $B(m,n,\lambda,\mu,q)$ is in $\N[q]$.
\end{remark}

\begin{remark}
	It is somewhat uncertain whether one can obtain \emph{any} information about the fibers of $\pr_1$. A natural attempt to find the size of $\pr_1^{-1}(W_1)$ would be to first fix an $\tl A$-submodule $\tl W_2$ of type $\lambda$ containing $\tl W_1:=\tl AW_1$, and then count the number of real structures $W_2$ of $\tl W_2$ that contain $W_1$. This count could potentially depend on the full structure of the \emph{flag} $\tl W_1\subeq \tl W_2$, not just on the type and cotype of $\tl W_1$ in $\tl W_2$. However, it is known that no elementary theory exists for the classification of $2$-step flags of DVR modules \cite{taitslin70}, making it unclear whether this refined counting problem is solvable even in principle. We propose the problem of seeking a more structural understanding of the fibers of $\pr_1$, including their geometry in the case $k=\R$ and $l=\C$.
\end{remark}

\subsection{Co-totally-real submodules}

This subsection is dual to the previous one. The development will be parallel and more concise.

As before, let $l/k$ be any quadratic field extension, and
\begin{align*}
	A=k[[T]]/T^m\subeq \tl A=l[[T]]/T^m.
\end{align*}
Let $\tl V$ be a finite-length $\tl A$-module. An $A$-submodule $W\subeq_A \tl V$ is \emph{spanning}, or \emph{co-totally-real} (cTR), if $\tl AW=\tl V$. Given a cTR submodule $W$ of $\tl V$, letting $\ut W=W\cap \Theta W$ be the largest $\tl A$-submodule contained in $W$, then $W/\ut W$ is a real structure of $\tl V/\ut W$. Consequently, the $\tl A$-cotype of $\ut W$ in $\tl V$ is the same as the $A$-cotype of $W$ in $\tl V$.

For a partition $\lambda$, define the \emph{cTR Grassmannian of cotype $\lambda$} by
\begin{equation*}
	\Gr^{\cTR}_A(\tl V;\lambda) := \big\{W\subeq_A \tl V \;:\; \text{$W$ cTR and $\type_A(\tl A/W)=\lambda$}\big\}.
\end{equation*}
For partitions $\lambda$ and $\mu$, define the \emph{cTR flag variety}
\begin{align*}
	\Fl^\cTR_A(\tl V;\lambda;\mu) := \big\{(W_2,W_1) \;: &\;\;\text{$W_2\subeq_A W_1\subeq_A \tl V$ with $W_1,W_2$ cTR},\\
	&\;\;\text{and $\type_A(\tl V/W_2)=\lambda$, $\type_A(\tl V/W_1)=\mu$}\big\}.
\end{align*}
Consider forgetful maps
\[\begin{tikzcd}[sep=scriptsize]
	& \Fl^{\cTR}_A(\tl V;\lambda;\mu) \\
	\Gr^{\cTR}_A(\tl V;\mu) && \Gr^{\cTR}_A(\tl V;\lambda)
	\arrow["{\pr_1}"', from=1-2, to=2-1]
	\arrow["{\pr_2}", from=1-2, to=2-3]
\end{tikzcd}.\]

\begin{lemma}\label{lem:co-tot-real}
	Assume the notation above, with $k=\Fq$, $l=\Fqsq$, and $\tl V=\tl A^n$. For partitions $\mu,\lambda\subeq (m^n)$,
	\begin{enumerate}[label={\textup{(\alph*)}},leftmargin=*,labelsep=0cm,align=left,itemsep=2pt]
		\item $\abs{\Gr^\cTR_A(\tl V;\lambda)}=g^{(m^n)}_\lambda(q^2) \dfrac{a_\lambda(q^2)}{a_\lambda(q)}$;
		
		\item Every fiber of $\pr_2$ has cardinality $g^\lambda_\mu(q)$;
		
		\item Every fiber of $\pr_1$ has the same cardinality.
	\end{enumerate}
\end{lemma}

\begin{proof}
	\hspace*{0pt}
	\begin{enumerate}[label={\textup{(\alph*)}},leftmargin=*,labelsep=0cm,align=left,itemsep=2pt]
		\item To pick $W\in \Gr^\cTR_A(\tl V;\lambda)$, we first choose $\ut W:=W\cap\Theta W$, which is an $\tl A$-submodule of $\tl V$ of cotype $\lambda$; there are $g^{(m^n)}_\lambda(q^2)$ choices. After we choose $\tl W$, picking $W$ amounts to picking a real structure of $\tl V/\ut W$, which has $a_\lambda(q^2)/a_\lambda(q)$ choices by Lemma~\ref{lem:real-struct-count}. This finishes the proof of part~(a).
		\item This follows from the simple observation that for $A$-submodules $W'\subeq W\subeq \tl V$, if $W$ is cTR, then so is $W'$.
		\item We note that $\pr_1$ is equivariant under the natural action of $G=\Aut_{\tl A}(\tl V)$. Since a cTR submodule $W$ of cotype $\mu$ is determined by an $\tl A$-submodule $\ut W$ of cotype $\mu$ and a real structure of $\tl V/\ut W$, it follows from Lemma~\ref{lem:transitive} and Lemma~\ref{lem:real-struct-count} that $G$ acts on $\Gr^\cTR_A(\tl V;\mu)$ transitively. The desired conclusion follows. \qedhere
	\end{enumerate}
\end{proof}

\begin{theorem}\label{thm:cTR-count}
	Let $k=\Fq$, $l=\Fqsq$, and $\tl V$ be an $l[[T]]$-module of type $(m^n)$. Then for partitions $\mu,\lambda\subeq (m^n)$, every cTR $k[[T]]$-submodules of $\tl V$ of cotype $\mu$ contains $B(m,n,\lambda,\mu,q)$ cTR $k[[T]]$-submodules of $\tl V$ of cotype $\lambda$, where
	\begin{equation}\label{eq:B-formula}
		B(m,n,\lambda,\mu,q)=\frac{g^{(m^n)}_\lambda(q^2) g^\lambda_\mu(q) a_\lambda(q^2)/a_\lambda(q)}{g^{(m^n)}_\mu(q^2) a_\mu(q^2)/a_\mu(q)}.
	\end{equation}
\end{theorem}
\begin{proof}
	This follows directly from Lemma~\ref{lem:co-tot-real}.
\end{proof}

\subsection{Remarks on $\nu^R_{\tl R^n}(s)$}

We use the techniques here to derive alternative formulas for some results in \cite{huang2025coh}.

\begin{proposition}[{cf.~\cite[Theorem~3.2(c)]{huang2025coh}}]
	For $R=R_{2,2m}'$, with $t=q^{-s}$, we have
	\begin{equation*}
		\nu^R_{\tl R^n}(s)=\sum_{\lambda} g^{(m^n)}_\lambda(q^2) \frac{a_\lambda(q^2)}{a_\lambda(q)} t^{\abs{\lambda}}.
	\end{equation*}
	\label{prop:rtilde}
\end{proposition}

\begin{proof}
	By \cite[Lemma~2.5]{huang2025coh}, we have
	\begin{equation*}
		\nu^R_{\tl R^n}(s)=\sum_{\lambda} \abs{\Gr_A^\cTR(\tl V;\lambda)}t^{\abs{\lambda}}.
	\end{equation*}
	The required formula then follows from Lemma~\ref{lem:co-tot-real}(a).
\end{proof}

\begin{remark}
	The result of \cite[Theorem~3.2(c)]{huang2025coh} is an $(m+1)$-fold multisum in indices $r$ and $(s_1,\dots,s_m)=:\lambda'$ (in the notation there), so the alternative formula here provides an ``$(m+1)\text{-fold} = m\text{-fold}$'' identity. One may ask what the nature of this identity is from a $q$-theoretic perspective. It turns out that with the indices $s_1,\dots,s_m$ fixed, the inner sum over $r$ simplifies to a one-term expression by the second $q$-Chu--Vandermonde sum, namely,
	\begin{align*}
		\sum_{r\ge 0} \frac{(-1)^r q^{r^2 + (1-2s_1)r} (q;q)_{2n-2r}}{(q^2;q^2)_r (q^2;q^2)_{n-r} (q;q)_{(n+s_1)-2r}} = \frac{(q;q)_{n-s_1}}{(q;q)_{n+s_1} (q^2;q^2)_{-s_1}}.
	\end{align*}
	Therefore, this ``$(m+1)\text{-fold} = m\text{-fold}$'' identity is essentially the above ``$1\text{-fold} = 0\text{-fold}$'' equality. This is not a coincidence: a closer look at the proof of \cite[Theorem~3.2(c)]{huang2025coh} will reveal that the sum over $r$ with a fixed $\lambda'$ essentially computes $\abs{\Gr_A^\cTR(\tl V;\lambda)}$ using the M\"obius inversion technique there, while Lemma~\ref{lem:co-tot-real}(a) directly gives the one-term expression it simplifies to. 
\end{remark}

Let us use Proposition~\ref{prop:rtilde} to give a simple, independent verification of the $s=0$ specialization of Theorem~\ref{thm:main}.

\begin{proposition}
	For $R=R_{2,2m}'$, with $t=q^{-s}$, we have
	\begin{equation*}
		\zetahat_{R,n}(0)=\frac{1}{(q^{-1};q^{-1})_n}\Br^{(2m+2)}_n(1,q^{-1}).
	\end{equation*}
\end{proposition}
\begin{proof}
	Recalling that $\dim_{\Fq} \tl R/R=m$, the formula \cite[eq.~(2.18)]{huang2025coh} reads $\nuhat_{R,n}(0)=q^{-mn^2}\nu^R_{\tl R^n}(0)$. Directly expanding Proposition~\ref{prop:rtilde}, setting $n_i=n-\lambda'_{m+1-i}$, and using
	$$\zetahat_{R,n}(s)=\zetahat_{\tl R,n}(s)\nuhat_{R,n}(s)=(t^2q^{-2};q^{-2})_n^{-1}\nuhat_{R,n}(s),$$
	the result follows.
\end{proof}

\subsection{The Quot zeta function}

We now put everything together. 

%\begin{lemma}[cf.~Lemma 9.8 of \cite{huangjiang2023torsionfree}]\label{lem:general-purpose-counting}
\begin{lemma}\label{lem:general-purpose-counting}
	Assume $k=\Fq$, $l=\Fqsq$, and $R=R_{2,2m}'$. Let $L_b$ be an $R$-lattice, and let $\mu$ be the type of the $A$-module $\tl R L_b/L_b$. The generating function
	\begin{equation*}
		G_{L_b}(t) := \sum_{L \in E_R(\tl R L_b; L_b)} \;t^{\dim_k L_b/L}
	\end{equation*}
	is given by
	\begin{equation}
		G_{L_b}(t)=\sum_\lambda B(m,n,\lambda,\mu,q) t^{\abs{\lambda}-\abs{\mu}}.
	\end{equation}
\end{lemma}

\begin{proof}
	Reducing modulo $\fc \tl RL_b$, we get
	\begin{equation*}
		G_{L_b}(t) = \sum_{W \in E_A(\tl V; W_b)} t^{\dim_k W_b/W},
	\end{equation*}
	where we denote $\tl V:=\tl RL_b/\fc \tl RL_b$ and $W_b:=\tl RL_b/L_b$. Note that $W_b\in E_A(\tl V)$, so $W_b$ is cTR in $\tl V$. By grouping the sum in terms of $\lambda:=\type_A(\tl V/W)$, Theorem~\ref{thm:cTR-count} gives the desired formula.
\end{proof}

We are ready for the first explicit formula for the Quot zeta function.
\begin{theorem}\label{thm:first-explicit}
	Assume $k=\Fq$, $l=\Fqsq$, and $R=R_{2,2m}'$. Then with $t:=q^{-s}$, we have
	\begin{equation}
		\nu^R_{R^n}(s)=(t^2;q^2)_{n} \zeta^R_{R^n}(s) = \sum_{\lambda,\mu} (t^2;q^2)_{\mu_1'} \,g^{(m^n)}_\mu(q) \,B(m,n,\lambda,\mu,q) \,t^{mn+\abs{\lambda}-2\abs{\mu}}.
	\end{equation}
\end{theorem}

\begin{proof}
	By Lemma~\ref{lem:boundary-lattice-invariants}(a), the formula in Theorem~\ref{thm:rationality-arithmetic} can be rewritten as
	\begin{align*}
		\nu^R_{R^n}(s) &= \sum_{L_b: \ut M\subeq L_b\subeq_R M} (t^2;q^2)_{\rk_1(\tl R L_b/\ut M)} \sum_{L \in E_R(\tl R L_b; L_b)} t^{\dim_k M/L}\\
		&=\sum_{L_b: \ut M\subeq L_b\subeq_R M} (t^2;q^2)_{\rk_1(\tl R L_b/\ut M)} t^{\dim_k M/L_b} G_{L_b}(t).
	\end{align*}
	Let $(m^n)-\mu$ be the cotype of $L_b$ in $M$. Since the type of $M/\ut M$ is $(m^n)$, the number of choices for $L_b$ is $g^{(m^n)}_\mu(q)$. We have $\dim_k M/L_b=mn-\abs{\mu}$. By Lemma~\ref{lem:boundary-lattice-invariants}, the $\tl A$-rank of $\tl R L_b/\ut M$ is $\mu'_1$, and the type of $\tl RL_b/L_b$ is $\mu$, so $G_{L_b}(t)$ is directly given by Lemma~\ref{lem:general-purpose-counting}. Putting everything together yields the desired formula.
\end{proof}

\begin{corollary}\label{coro:zeta-new}
	We have
	\begin{align}\label{eq:zeta-new}
		&\zetahat_{R'_{2,2m},n}(s)\notag\\
		&\quad= (z;z)_n \sum_{\substack{r_1,\ldots,r_m\ge 0\\s_1,\ldots,s_m\ge 0}} \frac{t^{\sum_{i=1}^m (2r_i-s_i)} z^{\sum_{i=1}^m (r_i^2-r_is_i+s_i^2)}}{(z;z)_{n-r_m} (z;z)_{r_{m}-r_{m-1}}\cdots (z;z)_{r_2-r_1}(z;z)_{r_1}(t^2z^2;z^2)_{r_1}}\notag\\
		&\quad\quad\times \frac{(-z;z)_{r_1}}{(-z;z)_{s_1}}\qbinom{r_m-s_{m-1}}{r_m-s_m}_z\qbinom{r_{m-1}-s_{m-2}}{r_{m-1}-s_{m-1}}_z\cdots \qbinom{r_{2}-s_1}{r_{2}-s_{2}}_z\qbinom{r_1}{r_1-s_1}_z,
	\end{align}
	where $z:=q^{-1}$ and $t:=q^{-s}$.
\end{corollary}

\begin{proof}
	Recall that
	\begin{align*}
		\zetahat_{R,n}(s)=\zetahat_{\tl R,n}(s)\nuhat_{R,n}(s)=(t^2q^{-2};q^{-2})_n^{-1} \nu^R_{R^n}(s+n).
	\end{align*}
	Thus, Theorem~\ref{thm:first-explicit} yields the desired formula after setting $r_i:=n-\lambda_{i}'$ and $s_i:=n-\mu_{i}'$ and applying the expressions recorded in \eqref{eq:g-formula} and \eqref{eq:B-formula}.
\end{proof}

\subsection{Proof of Theorem~\ref{thm:main}}

Assume Theorem~\ref{th:X-a-indep}, which will be established separately in the next section. We know from Corollary~\ref{coro:zeta-new} that
\begin{align*}
	\zetahat_{R'_{2,2m},n}(s) = (z;z)_n \cdot \frac{1}{(tz;z)_n}\cX_n^{(m)}(-1,-t,z).
\end{align*}
Now invoking \eqref{eq:X-nice-exp} gives
\begin{align*}
	&\zetahat_{R'_{2,2m},n}(s)\\
	&\quad = \frac{(z;z)_n}{(tz;z)_n} \sum_{n_1,\ldots,n_m\ge 0} \frac{t^{\sum_{i=1}^m 2n_i} z^{\sum_{i=1}^m n_i^2}}{(z;z)_{N-n_m} (z;z)_{n_{m}-n_{m-1}}\cdots (z;z)_{n_2-n_1}(z;z)_{n_1}(-tz;z)_{n_1}},
\end{align*}
which is precisely $(tz;z)_n^{-1}\Br_n^{(2m+2)}(t,z)$. \qed

\section{The phantasmal ``$a$''}\label{sec:a-indep}

Now the remaining and ultimate task is to establish the $a$-independence for our multisum $\cX_N^{(m)}(a,t,q)$ claimed in Theorem~\ref{th:X-a-indep}. Recall that
\begin{align}\label{eq:X-def-new}
	&\cX_N^{(m)}(a,t,q)\notag\\
	&\quad=(atq)_N\sum_{\substack{r_1,\ldots,r_m\ge 0\\s_1,\ldots,s_m\ge 0}} \frac{a^{\sum_{i=1}^m s_i}t^{\sum_{i=1}^m (2r_i-s_i)} q^{\sum_{i=1}^m (r_i^2-r_is_i+s_i^2)} (aq)_{r_1}}{(q)_{N-r_m} (q)_{r_{m}-r_{m-1}}\cdots (q)_{r_2-r_1}(q)_{r_1}(tq)_{r_1}(atq)_{r_1}(aq)_{s_1}}\notag\\
	&\quad\quad\times \qbinom{r_m-s_{m-1}}{r_m-s_m}\qbinom{r_{m-1}-s_{m-2}}{r_{m-1}-s_{m-1}}\cdots \qbinom{r_{2}-s_1}{r_{2}-s_{2}}\qbinom{r_1}{r_1-s_1}.
\end{align}

\subsection{$q$-Hypergeometric prerequisites}

We start by collecting some necessary $q$-hypergeometric transformations. Define the \emph{$q$-hypergeometric function ${}_{r}\phi_s$} by
\begin{align*}
	{}_{r}\phi_s\left(\begin{matrix} A_1,A_2,\ldots,A_r\\ B_1,B_2,\ldots,B_s \end{matrix}; q, z\right):=\sum_{n\ge 0} \frac{(A_1,A_2,\ldots,A_r;q)_n \big((-1)^n q^{\binom{n}{2}}\big)^{s-r+1} z^n}{(q,B_1,B_2,\ldots,B_{s};q)_n}.
\end{align*}

The \emph{$q$-binomial theorem} \cite[p.~354, eq.~(II.3)]{GR2004} is as follows:
\begin{lemma}[$q$-Binomial theorem]
	\begin{align}\label{eq:q-binomial}
		{}_{1} \phi_{0} \left(\begin{matrix}
			a\\
			-
		\end{matrix};q,z\right) = \frac{(az;q)_\infty}{(z;q)_\infty}.
	\end{align}
\end{lemma}

For ${}_2\phi_{1}$ series, we have the \emph{$q$-Gau\ss{} sum} \cite[p.~354, eq.~(II.8)]{GR2004}:
\begin{lemma}[$q$-Gau\ss{} sum]
	\begin{align}\label{eq:qGauss}
		{}_{2} \phi_{1} \left(\begin{matrix}
			a,b\\
			c
		\end{matrix};q,\frac{c}{ab}\right) = \frac{(c/a,c/b;q)_\infty}{(c,c/(ab);q)_\infty}.
	\end{align}
\end{lemma}

In particular, the $q$-Gau\ss{} sum specializes to the \emph{first $q$-Chu--Vandermonde sum} \cite[p.~354, eq.~(II.7)]{GR2004}:

\begin{lemma}[First $q$-Chu--Vandermonde sum]
	For any nonnegative integer $N$,
	\begin{align}\label{eq:qCV-1}
		{}_{2} \phi_{1} \left(\begin{matrix}
			a,q^{-N}\\
			c
		\end{matrix};q,\frac{cq^N}{a}\right) = \frac{(c/a;q)_N}{(c;q)_N}.
	\end{align}
\end{lemma}

Reversing the order of summation, we have the \emph{second $q$-Chu--Vandermonde sum} \cite[p.~354, eq.~(II.6)]{GR2004}:

\begin{lemma}[Second $q$-Chu--Vandermonde sum]
	For any nonnegative integer $N$,
	\begin{align}\label{eq:qCV-2}
		{}_{2} \phi_{1} \left(\begin{matrix}
			a,q^{-N}\\
			c
		\end{matrix};q,q\right) = \frac{a^N (c/a;q)_N}{(c;q)_N}.
	\end{align}
\end{lemma}

Next, we recall \emph{Heine's three transformations} \cite[p.~359, eqs.~(III.1--3)]{GR2004} for ${}_{2} \phi_{1}$ series:

\begin{lemma}[Heine's transformations]
	\begin{align}
		{}_{2} \phi_{1} \left(\begin{matrix}
			a,b\\
			c
		\end{matrix};q,z\right) &= \frac{(b,az;q)_\infty}{(c,z;q)_\infty} {}_{2} \phi_{1} \left(\begin{matrix}
			c/b,z\\
			az
		\end{matrix};q,b\right),\label{eq:Heine1}\\
		{}_{2} \phi_{1} \left(\begin{matrix}
			a,b\\
			c
		\end{matrix};q,z\right) &= \frac{(c/b,bz;q)_\infty}{(c,z;q)_\infty} {}_{2} \phi_{1} \left(\begin{matrix}
			abz/c,b\\
			bz
		\end{matrix};q,\frac{c}{b}\right),\label{eq:Heine2}\\
		{}_{2} \phi_{1} \left(\begin{matrix}
			a,b\\
			c
		\end{matrix};q,z\right) &= \frac{(abz/c;q)_\infty}{(z;q)_\infty} {}_{2} \phi_{1} \left(\begin{matrix}
			c/a,c/b\\
			c
		\end{matrix};q,\frac{abz}{c}\right).\label{eq:Heine3}
	\end{align}
\end{lemma}

The final important ingredient is a transformation formula for ${}_{3} \phi_{2}$ series \cite[p.~359, eq.~(III.9)]{GR2004}:

\begin{lemma}
	\begin{align}\label{eq:3phi2}
		{}_{3} \phi_{2} \left(\begin{matrix}
			a,b,c\\
			d,e
		\end{matrix};q,\frac{de}{abc}\right) = \frac{(e/a,de/(bc);q)_\infty}{(e,de/(abc);q)_\infty} {}_{3} \phi_{2} \left(\begin{matrix}
			a,d/b,d/c\\
			d,de/(bc)
		\end{matrix};q,\frac{e}{a}\right).
	\end{align}
\end{lemma}

\subsection{Reformulation}

In this part, our objective is to reformulate $\cX_N^{(m)}(a,t,q)$ in a way similar to that in \cite[Theorem~5.3]{shane2024multiple}.

\begin{theorem}
	For any nonnegative integer $N$,
	\begin{align}\label{eq:X-new}
		\cX_N^{(m)}(a,t,q)&= \frac{(aq)_\infty(t^2q)_\infty}{(tq)_\infty(atq^{N+1})_\infty} \sum_{\substack{u_1,\ldots,u_m\ge 0\\v_1,\ldots,v_m\ge 0}} a^{v_1+\sum_{i=1}^m u_i}t^{-2v_1+\sum_{i=1}^m (u_i+2v_i)}\notag\\
		&\quad\times \frac{q^{-v_1^2+v_1+\sum_{i=1}^m (u_i^2+u_iv_i+v_i^2)} (t)_{v_1}(a^{-1}t)_{v_1}}{(q)_{N-u_m}(t^2q)_{N+v_{m}}(q)_{u_m}(aq)_{u_1}(q)_{v_1}}\notag\\
		&\quad\times \qbinom{u_m}{u_{m-1}} \qbinom{u_{m-1}}{u_{m-2}}\cdots \qbinom{u_{2}}{u_{1}} \qbinom{v_1}{v_2} \cdots \qbinom{v_{m-2}}{v_{m-1}}\qbinom{v_{m-1}}{v_m}.
	\end{align}
\end{theorem}

We start by opening the $q$-binomial coefficients in \eqref{eq:X-def-new} and then reorganizing the $q$-factorials to get
\begin{align*}
	\cX_N^{(m)}(a,t,q)&= (atq)_N\sum_{\substack{r_1,\ldots,r_m\ge 0\\s_1,\ldots,s_m\ge 0}} \frac{a^{\sum_{i=1}^m s_i} t^{\sum_{i=1}^m (2r_i-s_i)} q^{\sum_{i=1}^m (r_i^2-r_is_i+s_i^2)} (aq)_{r_1}}{(q)_{N-r_m}(q)_{r_{m}-s_{m}}(q)_{s_m}(tq)_{r_1}(atq)_{r_1}(aq)_{s_1}}\notag\\
	&\quad\times \qbinom{s_m}{s_{m-1}}\cdots \qbinom{s_{2}}{s_{1}}\qbinom{r_m-s_{m-1}}{r_{m-1}-s_{m-1}}\cdots \qbinom{r_{2}-s_1}{r_1-s_{1}}.
\end{align*}
Now applying the substitution
\begin{align*}
	v_i := r_i - s_i
\end{align*}
for every $i$ with $1\le i\le m$, we have
\begin{align}\label{eq:X-sum-s-n}
	\cX_N^{(m)}(a,t,q)&= (atq)_N\sum_{\substack{s_1,\ldots,s_m\ge 0\\v_1,\ldots,v_m\ge 0}} \frac{(at)^{\sum_{i=1}^m s_i} t^{\sum_{i=1}^m 2v_i} q^{\sum_{i=1}^m (v_i^2+s_iv_i+s_i^2)} (aq)_{v_1+s_1}}{(q)_{(N-s_m)-v_m}(q)_{v_{m}}(q)_{s_m}(tq)_{v_1+s_1}(atq)_{v_1+s_1}(aq)_{s_1}}\notag\\
	&\quad\times \qbinom{s_m}{s_{m-1}}\cdots \qbinom{s_{2}}{s_{1}}\qbinom{v_m+s_m-s_{m-1}}{v_{m-1}}\cdots \qbinom{v_{2}+s_2-s_1}{v_1}.
\end{align}
For the moment, let us single out the sums over $v_1,\ldots,v_m$, which are
\begin{align*}
	\Sigma &:= \sum_{v_1,\ldots,v_m\ge 0} \frac{t^{\sum_{i=1}^m 2v_i} q^{\sum_{i=1}^m (v_i^2+s_iv_i)} (aq)_{v_1+s_1}}{(q)_{(N-s_m)-v_m}(q)_{v_{m}}(tq)_{v_1+s_1}(atq)_{v_1+s_1}}\notag\\
	&\ \quad\times\qbinom{v_m+s_m-s_{m-1}}{v_{m-1}}\cdots \qbinom{v_{2}+s_2-s_1}{v_1}\\
	&\ = \sum_{v_2,\ldots,v_m\ge 0} \frac{t^{\sum_{i=2}^m 2v_i} q^{\sum_{i=2}^m (v_i^2+s_iv_i)}}{(q)_{(N-s_m)-v_m}(q)_{v_{m}}}\qbinom{v_m+s_m-s_{m-1}}{v_{m-1}}\cdots \qbinom{v_{3}+s_3-s_2}{v_2}\\
	&\ \quad\times (q)_{v_2+s_2-s_1}\sum_{v_1\ge 0} \frac{t^{2v_1} q^{v_1^2+s_1v_1} (aq)_{v_1+s_1}}{(q)_{(v_2+s_2-s_1)-v_1} (q)_{v_1} (tq)_{v_1+s_1}(atq)_{v_1+s_1}}.
\end{align*}
To work on the inner sum over $v_1$, it is necessary to establish the following generalization of \cite[Lemma~5.1]{shane2024multiple}:

\begin{lemma}
	For any nonnegative integers $M$ and $N$,
	\begin{align}\label{eq:trans-2}
		&\sum_{n\ge 0} \frac{t^{2n} q^{n^2 + Mn} (aq)_{M+n}}{(q)_{N-n} (q)_n (tq)_{M+n} (atq)_{M+n}} = \frac{(aq)_\infty (t^2q)_\infty}{(tq)_\infty (atq)_\infty (q)_N} \sum_{n\ge 0} \frac{a^n q^{(M+1)n} (t)_{n} (a^{-1}t)_n}{(q)_{n} (t^2q)_{M+N+n}}.
	\end{align}
\end{lemma}

\begin{proof}
	We have
	\begin{align*}
		&\!\!\!\!\!\!\!\!\!\!\!\!\LHS\eqref{eq:trans-2}\\
		&= \frac{(aq)_M}{(q)_N(tq)_M(atq)_M} \sum_{n\ge 0} \frac{(-1)^n t^{2n} q^{\binom{n}{2}+(M+N+1)n} (q^{-N})_n (aq^{M+1})_n}{(q)_n (tq^{M+1})_n (atq^{M+1})_n}\\
		&= \frac{(aq)_M}{(q)_N(tq)_M(atq)_M} \lim_{\tau\to 0} {}_{3}\phi_{2} \left(\begin{matrix}
			q^{-N},1/\tau,aq^{M+1}\\
			tq^{M+1},atq^{M+1}
		\end{matrix};q,t^2q^{M+N+1}\tau\right)\\
		\text{\tiny (by \eqref{eq:3phi2})} &= \frac{(aq)_M}{(q)_N(tq)_M(atq)_M} \frac{(atq^{M+N+1})_\infty}{(atq^{M+1})_\infty} {}_{2}\phi_{1} \left(\begin{matrix}
			q^{-N},a^{-1}t\\
			tq^{M+1}
		\end{matrix};q,atq^{M+N+1}\right)\\
		\text{\tiny (by \eqref{eq:Heine2})}&= \frac{(aq)_M}{(q)_N (tq)_M (atq)_{M+N}} \frac{(aq^{M+1})_\infty (t^2q^{M+N+1})_\infty}{(tq^{M+1})_\infty (atq^{M+N+1})_\infty} {}_{2}\phi_{1} \left(\begin{matrix}
			t,a^{-1}t\\
			t^2q^{M+N+1}
		\end{matrix};q,aq^{M+1}\right)\\
		&= \frac{(aq)_\infty (t^2q)_\infty}{(tq)_\infty (atq)_\infty (q)_N} \sum_{n\ge 0} \frac{a^n q^{(M+1)n} (t)_{n} (a^{-1}t)_n}{(q)_{n} (t^2q)_{M+N+n}},
	\end{align*}
	as claimed.
\end{proof}

Applying this lemma, it follows that
\begin{align*}
	\Sigma& = \sum_{v_2,\ldots,v_m\ge 0} \frac{t^{\sum_{i=2}^m 2v_i} q^{\sum_{i=2}^m (v_i^2+s_iv_i)}}{(q)_{(N-s_m)-v_m}(q)_{v_{m}}}\qbinom{v_m+s_m-s_{m-1}}{v_{m-1}}\cdots \qbinom{v_{3}+s_3-s_2}{v_2}\\
	&\quad\times \frac{(aq)_\infty (t^2q)_\infty}{(tq)_\infty (atq)_\infty} \sum_{v_1\ge 0} \frac{a^{v_1} q^{(s_1+1)v_1} (t)_{v_1} (a^{-1}t)_{v_1}}{(q)_{v_1} (t^2q)_{v_1+v_2+s_2}},
\end{align*}
which, by interchanging the sum over $v_1$ and the remaining sums, gives us that
\begin{align*}
	\Sigma& = \frac{(aq)_\infty (t^2q)_\infty}{(tq)_\infty (atq)_\infty (q)_{N-s_m}} \sum_{v_1\ge 0} \frac{a^{v_1} q^{(s_1+1) v_1} (t)_{v_1} (a^{-1}t)_{v_1}}{(q)_{v_1}}\\
	&\quad\times \sum_{v_2,\ldots,v_m\ge 0} \frac{t^{\sum_{i=2}^m 2v_i} q^{\sum_{i=2}^m (v_i^2+s_iv_i)}}{(t^2q)_{v_1+v_2+s_2}}\\
	&\quad\times \qbinom{N-s_m}{v_m} \qbinom{v_m+s_m-s_{m-1}}{v_{m-1}}\cdots \qbinom{v_{3}+s_3-s_2}{v_2}.
\end{align*}
Now we recall that a consequence of \cite[Lemma~5.2]{shane2024multiple} is
\begin{align*}
	&\sum_{v_2,\ldots,v_m\ge 0} \frac{t^{\sum_{i=2}^m 2v_i} q^{\sum_{i=2}^m (v_i^2+s_iv_i)}}{(t^2q)_{v_1+v_2+s_2}} \qbinom{N-s_m}{v_m}\qbinom{v_m+s_m-s_{m-1}}{v_{m-1}}\cdots \qbinom{v_{3}+s_3-s_2}{v_{2}}\\
	&\qquad = \sum_{v_2,\ldots,v_m\ge 0} \frac{t^{\sum_{i=2}^m 2v_i} q^{\sum_{i=2}^m (v_i^2+s_iv_i)}}{(t^2q)_{N+v_m}} \qbinom{v_1}{v_2} \cdots \qbinom{v_{m-1}}{v_m},
\end{align*}
so that
\begin{align*}
	\Sigma &= \frac{(aq)_\infty (t^2q)_\infty}{(tq)_\infty (atq)_\infty (q)_{N-s_m}} \sum_{v_1,\ldots,v_m\ge 0} a^{v_1} t^{\sum_{i=2}^m 2v_i} q^{v_1+s_1v_1+\sum_{i=2}^m (v_i^2+s_iv_i)}\\
	&\quad\times\frac{(t)_{v_1} (a^{-1}t)_{v_1}}{(q)_{v_{1}} (t^2q)_{N+v_m}}\qbinom{v_1}{v_2}_q \cdots \qbinom{v_{m-1}}{v_m}_q.
\end{align*}
Finally, we substitute the above relation into \eqref{eq:X-sum-s-n} and rename $s_i$ by $u_i$ for each $i$ with $1\le i\le m$, thereby concluding \eqref{eq:X-new}.

\subsection{An auxiliary series}

From the moment, we focus on the following auxiliary series:
\begin{align}\label{eq:V-def}
	\cV_N^{(m)}(a,t,q)&:= \sum_{\substack{u_1,\ldots,u_m\ge 0\\v_1,\ldots,v_m\ge 0}} a^{v_1+\sum_{i=1}^m u_i}t^{-2v_1+\sum_{i=1}^m (u_i+2v_i)} q^{-v_1^2+v_1+\sum_{i=1}^m (u_i^2+u_iv_i+v_i^2)}\notag\\
	&\ \quad\times \frac{ (t)_{v_1}(a^{-1}t)_{v_1}}{(q)_{N-u_m}(q)_{u_m}(aq)_{u_1}(q)_{v_1}} \qbinom{u_m}{u_{m-1}}\cdots \qbinom{u_{2}}{u_{1}} \qbinom{v_1}{v_2} \cdots \qbinom{v_{m-1}}{v_m}.
\end{align}
An important feature of this family of series is its recursive structure.

\begin{lemma}\label{le:V-rec}
	For $m\ge 2$,
	\begin{align}\label{eq:V-rec}
		\cV_N^{(m)}(a,t,q) = \sum_{u,v\ge 0} \frac{a^{u+v} t^{u+2(m-1)v} q^{u^2+(u+1)v + (m-1)v^2} (t)_v (a^{-1}t)_v}{(q)_{N-u} (q)_v} \cV_u^{(m-1)}(a,t q^v,q).
	\end{align}
\end{lemma}

\begin{proof}
	In \eqref{eq:V-def}, we first open the $q$-binomial coefficients and get
	\begin{align*}
		&\cV_N^{(m)}(a,t,q)\\
		&\quad= \sum_{\substack{u_1,\ldots,u_m\ge 0\\v_1,\ldots,v_m\ge 0}} a^{v_1+\sum_{i=1}^m u_i}t^{-2v_1+\sum_{i=1}^m (u_i+2v_i)} q^{-v_1^2+v_1+\sum_{i=1}^m (u_i^2+u_iv_i+v_i^2)}\\
		&\quad\quad\times \frac{ (t)_{v_1}(a^{-1}t)_{v_1}}{(q)_{N-u_m}(q)_{v_m}(q)_{u_1}(aq)_{u_1}} \frac{1}{(q)_{u_m-u_{m-1}}\cdots (q)_{u_2-u_1} (q)_{v_1-v_2}\cdots (q)_{v_{m-1}-v_m}}.
	\end{align*}
	Singling out the sums over $u_m$ and $v_m$, then $\cV_N^{(m)}(a,t,q)$ equals
	\begin{align*}
		&\sum_{u_m,v_m\ge 0} \frac{a^{u_m} t^{2v_m} q^{u_m^2+u_m v_m +v_m^2}}{(q)_{N-u_m}(q)_{v_m}}\\
		&\times \sum_{\substack{u_1,\ldots,u_{m-1}\ge 0\\v_1,\ldots,v_{m-1}\ge v_m}} a^{v_1+\sum_{i=1}^{m-1} u_i}t^{-2v_1+\sum_{i=1}^{m-1} (u_i+2v_i)} q^{-v_1^2+v_1+\sum_{i=1}^{m-1} (u_i^2+u_iv_i+v_i^2)}\\
		&\times \frac{ (t)_{v_1}(a^{-1}t)_{v_1}}{(q)_{u_1}(aq)_{u_1}} \frac{1}{(q)_{u_m-u_{m-1}}\cdots (q)_{u_2-u_1} (q)_{v_1-v_2}\cdots (q)_{v_{m-1}-v_m}}.
	\end{align*}
	For each $i$ with $1\le i\le m-1$, we shift the index
	\begin{align*}
		v_i \mapsto v_i + v_m,
	\end{align*}
	and then \eqref{eq:V-rec} follows after simplification.
\end{proof}

Now our task is to simplify $\cV_N^{(m)}(a,t,q)$ to an $m$-fold sum.

\begin{theorem}
	For any nonnegative integer $N$,
	\begin{align}\label{eq:V-expression}
		\cV_N^{(m)}(a,t,q) &= \frac{(atq)_\infty}{(aq)_\infty(q)_N} \sum_{n_1,\ldots,n_m\ge 0} (-1)^{n_m} t^{-n_m+\sum_{i=1}^m 2n_i} q^{-\binom{n_m}{2}+\sum_{i=1}^m n_i^2} \notag\\
		&\quad\times  \frac{(t)_{n_m}}{(q)_{n_m}(atq)_{N+n_1}} \qbinom{n_m}{n_{m-1}}\qbinom{n_{m-1}}{n_{m-2}}\cdots \qbinom{n_2}{n_1}.
	\end{align}
\end{theorem}

\begin{proof}[Proof of the base case]
	Recall that
	\begin{align*}
		\cV_N^{(1)}(a,t,q) &= \sum_{u_1,v_1\ge 0} \frac{a^{u_1+v_1} t^{u_1} q^{u_1^2 + u_1v_1 + v_1} (t)_{v_1} (a^{-1}t)_{v_1}}{(q)_{N-u_1} (q)_{u_1} (aq)_{u_1} (q)_{v_1}}\\
		&= \sum_{v_1\ge 0} \frac{a^{v_1} q^{v_1} (t)_{v_1} (a^{-1}t)_{v_1}}{(q)_{v_1}} \sum_{u_1\ge 0} \frac{a^{u_1} t^{u_1} q^{u_1^2 + v_1 u_1}}{(q)_{N-u_1} (q)_{u_1} (aq)_{u_1}}.
	\end{align*}
	For the inner sum over $u_1$, we have
	\begin{align*}
		\sum_{u_1\ge 0} \frac{a^{u_1} t^{u_1} q^{u_1^2 + v_1 u_1}}{(q)_{N-u_1} (q)_{u_1} (aq)_{u_1}} = \frac{1}{(q)_N} \lim_{\tau\to 0} {}_{2}\phi_{1} \left(\begin{matrix}
			1/\tau,q^{-N}\\
			aq
		\end{matrix};q,atq^{N+v_1+1}\tau\right).
	\end{align*}
	We temporarily assume $|t|<1$ to ensure the convergence condition for the application of Heine's third transformation \eqref{eq:Heine3} to the ${}_{2}\phi_{1}$ series, especially when $v_1 = 0$. Then,
	\begin{align*}
		\sum_{u_1\ge 0} \frac{a^{u_1} t^{u_1} q^{u_1^2 + v_1 u_1}}{(q)_{N-u_1} (q)_{u_1} (aq)_{u_1}} &= \frac{1}{(q)_N}\cdot (tq^{v_1})_\infty {}_{2}\phi_{1} \left(\begin{matrix}
			0,aq^{N+1}\\
			aq
		\end{matrix};q,tq^{v_1}\right)\\
		&= \frac{(t)_\infty}{(q)_N (t)_{v_1}} \sum_{u_1\ge 0} \frac{t^{u_1}q^{u_1v_1}(aq^{N+1})_{u_1}}{(q)_{u_1} (aq)_{u_1}}.
	\end{align*}
	It follows that
	\begin{align*}
		\cV_N^{(1)}(a,t,q) &= \frac{(t)_\infty}{(q)_N} \sum_{u_1\ge 0} \frac{t^{u_1} (aq^{N+1})_{u_1}}{(q)_{u_1} (aq)_{u_1}} \sum_{v_1\ge 0} \frac{a^{v_1} q^{(u_1+1)v_1} (a^{-1}t)_{v_1}}{(q)_{v_1}}\\
		\text{\tiny (by \eqref{eq:q-binomial})}&=  \frac{(t)_\infty}{(q)_N} \sum_{u_1\ge 0} \frac{t^{u_1} (aq^{N+1})_{u_1}}{(q)_{u_1} (aq)_{u_1}} \frac{(tq^{u_1+1})_\infty}{(aq^{u_1+1})_\infty}\\
		&= \frac{(t)_\infty (tq)_\infty}{(aq)_\infty (q)_N} \sum_{u_1\ge 0} \frac{t^{u_1} (aq^{N+1})_{u_1}}{(q)_{u_1} (tq)_{u_1}}\\
		&= \frac{(t)_\infty (tq)_\infty}{(aq)_\infty (q)_N} \lim_{\tau\to 0} {}_{2}\phi_{1} \left(\begin{matrix}
			aq^{N+1},tq\tau\\
			tq
		\end{matrix};q,t\right)\\
		\text{\tiny (by \eqref{eq:Heine1})} &= \frac{(t)_\infty (tq)_\infty}{(aq)_\infty (q)_N} \lim_{\tau\to 0} \frac{(tq\tau)_\infty(atq^{N+1})_\infty}{(tq)_\infty(t)_\infty} {}_{2}\phi_{1} \left(\begin{matrix}
			1/\tau,t\\
			atq^{N+1}
		\end{matrix};q,tq\tau\right)\\
		&= \frac{(atq)_\infty}{(aq)_\infty (q)_N} \sum_{n_1\ge 0} \frac{(-1)^{n_1} t^{n_1} q^{\binom{n_1+1}{2}} (t)_{n_1}}{(q)_{n_1} (atq)_{N+n_1}}.
	\end{align*}
	Finally, this relation can be analytically continued from $|t|<1$, as assumed earlier, to $t\in \mathbb{C}$. Therefore, we arrive at \eqref{eq:V-expression} for $m=1$.
\end{proof}

\begin{proof}[Proof of the inductive step]
	Now let us assume \eqref{eq:V-expression} for a certain $m-1$ with $m\ge 2$ so that
	\begin{align*}
		\cV_u^{(m-1)}(a,tq^v,q) &= \frac{(atq^{v+1})_\infty}{(aq)_\infty(q)_u} \sum_{n_1,\ldots,n_{m-1}\ge 0} (-1)^{n_{m-1}} t^{-n_{m-1}+\sum_{i=1}^{m-1} 2n_i}\\
		&\quad\times  \frac{q^{-\binom{n_{m-1}}{2}-vn_{m-1}+\sum_{i=1}^{m-1} (n_i^2 + 2vn_i)}(tq^v)_{n_{m-1}}}{(q)_{n_{m-1}-n_{m-2}}\cdots (q)_{n_2-n_1}(q)_{n_{1}}(atq^{v+1})_{u+n_1}}.
	\end{align*}
	Applying this relation to \eqref{eq:V-rec}, we have
	\begin{align*}
		\cV_N^{(m)}(a,t,q) &= \frac{(atq)_\infty}{(aq)_\infty} \sum_{u,v\ge 0} \frac{(-1)^v (at)^{u+v} q^{u^2+uv+\binom{v+1}{2}} (a^{-1}t)_v}{(q)_{N-u} (q)_u (q)_v}\\
		&\quad\times \sum_{n_1,\ldots,n_{m-1}\ge 0} (-1)^{n_{m-1}+v} t^{-(n_{m-1}+v) + \sum_{i=1}^{m-1} 2(n_i+v)}\\
		&\quad\times \frac{q^{-\binom{n_{m-1}+v}{2}+\sum_{i=1}^{m-1} (n_i+v)^2} (t)_{n_{m-1}+v}}{(q)_{n_{m-1}-n_{m-2}}\cdots (q)_{n_2-n_1}(q)_{n_{1}} (atq)_{u+(n_1+v)}}.
	\end{align*}
	For each $i$ with $1\le i\le m-1$, we make the substitution $l_i := n_i+v$ and derive that
	\begin{align*}
		\cV_N^{(m)}(a,t,q) &= \frac{(atq)_\infty}{(aq)_\infty} \sum_{l_1,\ldots,l_{m-1}\ge 0} \frac{(-1)^{l_{m-1}} t^{-l_{m-1} + \sum_{i=1}^{m-1} 2l_i} q^{-\binom{l_{m-1}}{2}+\sum_{i=1}^{m-1} l_i^2} (t)_{l_{m-1}}}{(q)_{l_{m-1}-l_{m-2}}\cdots (q)_{l_2-l_1}}\\
		&\quad\times \sum_{u,v\ge 0} \frac{(-1)^v (at)^{u+v} q^{u^2+uv+\binom{v+1}{2}} (a^{-1}t)_v}{(q)_{N-u} (q)_{l_{1}-v} (atq)_{l_1+u} (q)_u (q)_v}.
	\end{align*}
	It remains to show that
	\begin{align}\label{eq:induction-key}
		\sum_{u,v\ge 0} \frac{(-1)^v (at)^{u+v} q^{u^2+uv+\binom{v+1}{2}} (a^{-1}t)_v}{(q)_{N-u} (q)_{l_{1}-v} (atq)_{l_1+u} (q)_u (q)_v} = \frac{1}{(q)_N} \sum_{l_0\ge 0} \frac{t^{2l_0} q^{l_0^2}}{(q)_{l_1-l_0} (q)_{l_0} (atq)_{N+l_0}}.
	\end{align}
	Once it has been done, we only need to rename the indices by $l_i\mapsto n_{i+1}$ for $0\le i\le m-1$ to conclude \eqref{eq:V-expression}. Now for the proof of \eqref{eq:induction-key}, we proceed in a similar way to that for \cite[eq.~(6.4)]{shane2024multiple}. Note that
	\begin{align*}
		\LHS\eqref{eq:induction-key} = \sum_{u\ge 0} \frac{(at)^{u} q^{u^2}}{(q)_{N-u} (atq)_{l_1+u} (q)_u} \sum_{v\ge 0} \frac{(-1)^v (at)^{v} q^{\binom{v}{2}+(u+1)v} (a^{-1}t)_v}{(q)_{l_{1}-v} (q)_v}.
	\end{align*}
	For the inner sum, we have
	\begin{align*}
		\sum_{v\ge 0} \frac{(-1)^v (at)^{v} q^{\binom{v}{2}+(u+1)v} (a^{-1}t)_v}{(q)_{l_{1}-v} (q)_v} &= \frac{1}{(q)_{l_1}} \lim_{\tau\to 0} {}_{2}\phi_{1} \left(\begin{matrix}
			a^{-1}t,q^{-l_1}\\
			t^2q^{u+1}\tau
		\end{matrix};q,atq^{l_1+u+1}\right)\\
		\text{\tiny (by \eqref{eq:Heine2})} &= \frac{(atq^{u+1})_\infty}{(atq^{l_1+u+1})_\infty} \sum_{l_0\ge 0} \frac{t^{2l_0}q^{l_0^2+ul_0}}{(q)_{l_1-l_0}(q)_{l_0}(atq^{u+1})_{l_0}}.
	\end{align*}
	Therefore,
	\begin{align*}
		\LHS\eqref{eq:induction-key} &= \sum_{l_0\ge 0} \frac{t^{2l_0}q^{l_0^2}}{(q)_{l_1-l_0}(q)_{l_0}} \sum_{u\ge 0} \frac{(at)^{u} q^{u^2}}{(q)_{N-u} (atq)_{l_0+u} (q)_u}\\
		&= \frac{1}{(q)_N} \sum_{l_0\ge 0} \frac{t^{2l_0}q^{l_0^2}}{(q)_{l_1-l_0}(q)_{l_0} (atq)_{l_0}} \lim_{\tau\to 0} {}_{2}\phi_{1} \left(\begin{matrix}
			1/\tau,q^{-N}\\
			atq^{l_0+1}
		\end{matrix};q,atq^{N+l_0+1}\tau\right)\\
		\text{\tiny (by \eqref{eq:qCV-1})} &= \frac{1}{(q)_N} \sum_{l_0\ge 0} \frac{t^{2l_0}q^{l_0^2}}{(q)_{l_1-l_0}(q)_{l_0} (atq)_{l_0}} \frac{1}{(atq^{l_0+1};q)_N},
	\end{align*}
	as requested.
\end{proof}

\subsection{The $a$-independence}

From now on, we complete the proof of Theorem~\ref{th:X-a-indep}. To achieve this goal, we consider the cases $m=1$ and $m\ge 2$ separately.

\subsubsection{The case $m=1$}

In view of \eqref{eq:X-new},
\begin{align*}
	&\cX_N^{(1)}(a,t,q)\\
	&\quad= \frac{(aq)_\infty(t^2q)_\infty}{(tq)_\infty(atq^{N+1})_\infty}\sum_{u_1,v_1\ge 0} \frac{a^{u_1+v_1} t^{u_1} q^{u_1^2 + u_1v_1 + v_1} (t)_{v_1} (a^{-1}t)_{v_1}}{(q)_{N-u_1} (t^2q)_{N+v_1} (q)_{u_1} (aq)_{u_1} (q)_{v_1}}\\
	&\quad= \frac{(aq)_\infty(t^2q)_\infty}{(tq)_\infty(atq^{N+1})_\infty}\frac{1}{(t^2q)_N} \sum_{u_1\ge 0} \frac{a^{u_1} t^{u_1} q^{u_1^2}}{(q)_{N-u_1} (q)_{u_1} (aq)_{u_1}} {}_{2}\phi_{1} \left(\begin{matrix}
		t,a^{-1}t\\
		t^2q^{N+1}
	\end{matrix};q,aq^{u_1+1}\right).
\end{align*}
For the ${}_{2}\phi_{1}$ series, we apply Heine's second transformation \eqref{eq:Heine2} and find that it equals
\begin{align*}
	\frac{(atq^{N+1})_\infty(tq^{u_1+1})_\infty}{(t^2q^{N+1})_\infty(aq^{u_1+1})_\infty} \sum_{v_1\ge 0} \frac{(at)^{v_1} q^{(N+1)v_1} (q^{u_1-N})_{v_1}(a^{-1}t)_{v_1}}{(q)_{v_1}(tq^{u_1+1})_{v_1}}.
\end{align*}
Thus,
\begin{align*}
	\cX_N^{(1)}(a,t,q) = \sum_{u_1,v_1\ge 0} \frac{(-1)^{v_1} (at)^{u_1+v_1} q^{u_1^2 + u_1v_1 + \binom{v_1+1}{2}} (a^{-1}t)_{v_1}}{(q)_{N-u_1-v_1} (q)_{u_1} (q)_{v_1} (tq)_{u_1+v_1}}.
\end{align*}
To see the $a$-independence, we shall show
\begin{align}\label{eq:X1-key}
	\sum_{u,v\ge 0} \frac{(-1)^{v} (at)^{u+v} q^{u^2 + uv + \binom{v+1}{2}} (a^{-1}t)_{v}}{(q)_{N-u-v} (q)_{u} (q)_{v} (tq)_{u+v}} = \sum_{n\ge 0} \frac{t^{2n} q^{n^2}}{(q)_{N-n}(q)_n(tq)_n}.
\end{align}
Making the change of indices $n:=u+v$, we have
\begin{align*}
	\LHS\eqref{eq:X1-key} &= \sum_{n\ge 0} \frac{(at)^n q^{n^2}}{(q)_{N-n}(tq)_n} \sum_{v\ge 0} \frac{(-1)^v q^{\binom{v}{2} + (1-n)v} (a^{-1}t)_v}{(q)_{n-v} (q)_v}\\
	&= \sum_{n\ge 0} \frac{(at)^n q^{n^2}}{(q)_{N-n}(tq)_n} \frac{1}{(q)_n} {}_{2}\phi_{1} \left(\begin{matrix}
		a^{-1}t, q^{-n}\\
		0
	\end{matrix};q,q\right)\\
	\text{\tiny (by \eqref{eq:qCV-1})} &= \sum_{n\ge 0} \frac{t^{2n} q^{n^2}}{(q)_{N-n}(q)_n(tq)_n},
\end{align*}
as desired.

\subsubsection{The cases $m\ge 2$}

We start by noticing that through the same strategy as we prove Lemma~\ref{le:V-rec}, the following relation is clear.

\begin{lemma}\label{le:X-rec}
	For $m\ge 2$,
	\begin{align}\label{eq:X-rec}
		\cX_N^{(m)}(a,t,q)&= \frac{(aq)_\infty(t^2q)_\infty}{(tq)_\infty(atq^{N+1})_\infty} \sum_{u,v\ge 0} a^{u+v} t^{u+2(m-1)v} q^{u^2+(u+1)v + (m-1)v^2} \notag\\
		&\quad\times\frac{(t)_v (a^{-1}t)_v}{(q)_{N-u} (t^2q)_{N+v} (q)_v} \cV_u^{(m-1)}(a,t q^v,q).
	\end{align}
\end{lemma}

Now we plug \eqref{eq:V-expression} into \eqref{eq:X-rec} and further change the indices $l_i := n_i+v$ for each $1\le i\le m-1$. Thus,
\begin{align}\label{eq:X-rec-new}
	&\cX_N^{(m)}(a,t,q)\notag\\
	&\quad= \frac{(t^2q)_\infty(atq)_N}{(tq)_\infty} \sum_{l_1,\ldots,l_{m-1}\ge 0} \frac{(-1)^{l_{m-1}} t^{-l_{m-1} + \sum_{i=1}^{m-1} 2l_i} q^{-\binom{l_{m-1}}{2}+\sum_{i=1}^{m-1} l_i^2} (t)_{l_{m-1}}}{(q)_{l_{m-1}-l_{m-2}}\cdots (q)_{l_2-l_1}}\notag\\
	&\quad\quad\times \sum_{u,v\ge 0} \frac{(-1)^v (at)^{u+v} q^{u^2+uv+\binom{v+1}{2}} (a^{-1}t)_v}{(q)_{N-u} (t^2q)_{N+v} (q)_{l_{1}-v} (atq)_{l_1+u} (q)_u (q)_v}.
\end{align}
To see the $a$-independence, it is sufficient to show
\begin{align}\label{eq:a-indep-key}
	&(atq)_N \sum_{u,v\ge 0} \frac{(-1)^v (at)^{u+v} q^{u^2+uv+\binom{v+1}{2}} (a^{-1}t)_v}{(q)_{N-u} (t^2q)_{N+v} (q)_{l_{1}-v} (atq)_{l_1+u} (q)_u (q)_v}\notag\\
	&\qquad\qquad\qquad\qquad\qquad\qquad\qquad = \frac{1}{(t^2q)_{N+l_1}} \sum_{l_0\ge 0} \frac{t^{2l_0} q^{l_0^2}}{(q)_{N-l_0} (q)_{l_1-l_0} (q)_{l_0}}.
\end{align}
Note that
\begin{align*}
	\LHS\eqref{eq:a-indep-key} = (atq)_N \sum_{u\ge 0} \frac{(at)^{u} q^{u^2}}{(q)_{N-u} (atq)_{l_1+u} (q)_u} \sum_{v\ge 0} \frac{(-1)^v (at)^{v} q^{\binom{v}{2}+(u+1)v} (a^{-1}t)_v}{(q)_{l_{1}-v} (q)_v (t^2q)_{N+v}}.
\end{align*}
For the inner sum, we have
\begin{align*}
	&\!\!\!\!\!\!\!\!\sum_{v\ge 0} \frac{(-1)^v (at)^{v} q^{\binom{v}{2}+(u+1)v} (a^{-1}t)_v}{(q)_{l_{1}-v} (q)_v (t^2q)_{N+v}}\\
	&= \frac{1}{(q)_{l_1} (t^2q)_N} {}_{2}\phi_{1} \left(\begin{matrix}
		a^{-1}t,q^{-l_1}\\
		t^2q^{N+1}
	\end{matrix};q,atq^{l_1+u+1}\right)\\
	\text{\tiny (by \eqref{eq:Heine2})} &= \frac{(atq^{u+1})_\infty (q)_{N-u}}{(atq^{l_1+u+1})_\infty (t^2q)_{N+l_1}} \sum_{v\ge 0} \frac{t^{2v} q^{v^2+uv}}{(q)_{N-u-v}(q)_{l_1-v}(q)_v(atq^{u+1})_{v}}.
\end{align*}
Therefore,
\begin{align*}
	\LHS\eqref{eq:a-indep-key} = \frac{(atq)_N}{(t^2q)_{N+l_1}} \sum_{u,v\ge 0} \frac{a^u t^{u+2v} q^{u^2+uv+v^2}}{(q)_{N-u-v} (q)_{l_1-v} (q)_u (q)_v (atq)_{u+v}}.
\end{align*}
Now we single out the sum over $u$,
\begin{align*}
	\LHS\eqref{eq:a-indep-key} = \frac{(atq)_N}{(t^2q)_{N+l_1}} \sum_{v\ge 0} \frac{t^{2v} q^{v^2}}{(q)_{l_1-v} (q)_v (atq)_{v}} \sum_{u\ge 0} \frac{(at)^u q^{u^2 + vu}}{(q)_{(N-v)-u} (q)_u (atq^{v+1})_u}.
\end{align*}
Then,
\begin{align*}
	\sum_{u\ge 0} \frac{(at)^u q^{u^2 + vu}}{(q)_{(N-v)-u} (q)_u (atq^{v+1})_u} &= \frac{1}{(q)_{N-v}} \lim_{\tau\to 0} {}_{2}\phi_{1} \left(\begin{matrix}
		1/\tau,q^{-(N-v)}\\
		atq^{v+1}
	\end{matrix};q,atq^{N+1}\tau\right)\\
	\text{\tiny (by \eqref{eq:qCV-1})} &= \frac{1}{(q)_{N-v} (atq^{v+1})_{N-v}}.
\end{align*}
It follows that
\begin{align*}
	\LHS\eqref{eq:a-indep-key} = \frac{(atq)_N}{(t^2q)_{N+l_1}} \sum_{v\ge 0} \frac{t^{2v} q^{v^2}}{(q)_{l_1-v} (q)_v (atq)_{v}} \frac{1}{(q)_{N-v} (atq^{v+1})_{N-v}},
\end{align*}
which gives \eqref{eq:a-indep-key} by renaming $v$ by $l_0$. Finally, we apply \eqref{eq:a-indep-key} to \eqref{eq:X-rec-new} and arrive at the following $a$-independent expression for $\cX_N^{(m)}(a,t,q)$.

\begin{theorem}
	For $m\ge 2$,
	\begin{align}\label{eq:X-exp-2}
		\cX_N^{(m)}(a,t,q) &= \frac{(t^2q)_\infty}{(tq)_\infty} \sum_{n_1,\ldots,n_m\ge 0} (-1)^{n_m} t^{-n_m+\sum_{i=1}^m 2n_i} q^{-\binom{n_m}{2}+\sum_{i=1}^m n_i^2} \notag\\
		&\quad\times \frac{(t)_{n_m}}{(t^2q)_{N+n_2}(q)_{N-n_1}(q)_{n_m}} \qbinom{n_m}{n_{m-1}}\qbinom{n_{m-1}}{n_{m-2}}\cdots \qbinom{n_2}{n_1}.
	\end{align}
\end{theorem}

\subsection{Further implications}

With the proof of the $a$-independence for our series $\cX_N^{(m)}(a,t,q)$ complete, we consider some further implications. Clearly, there is a free choice of $a$, and in particular, taking $a=0$ implies the nice expression in \eqref{eq:X-nice-exp},
\begin{align}\label{eq:X-nice-exp-repeat}
	\cX_N^{(m)}(a,t,q) = \sum_{n_1,\ldots,n_m\ge 0} \frac{t^{\sum_{i=1}^m 2n_i} q^{\sum_{i=1}^m n_i^2}}{(q)_{N-n_m} (q)_{n_{m}-n_{m-1}}\cdots (q)_{n_2-n_1}(q)_{n_1}(tq)_{n_1}}.
\end{align}
Combining \eqref{eq:X-exp-2} and \eqref{eq:X-nice-exp-repeat}, we have the following intriguing relation.

\begin{corollary}
	For $m\ge 2$,
	\begin{align}
		&\sum_{n_1,\ldots,n_m\ge 0} \frac{(-1)^{n_m} t^{-n_m+\sum_{i=1}^m 2n_i} q^{-\binom{n_m}{2}+\sum_{i=1}^m n_i^2} (t)_{n_m}}{(t^2q)_{N+n_2}(q)_{N-n_1}(q)_{n_m}} \qbinom{n_m}{n_{m-1}}\qbinom{n_{m-1}}{n_{m-2}}\cdots \qbinom{n_2}{n_1}\notag\\
		&\quad = \frac{(tq)_\infty}{(t^2q)_\infty} \sum_{n_1,\ldots,n_m\ge 0} \frac{t^{\sum_{i=1}^m 2n_i} q^{\sum_{i=1}^m n_i^2}}{(q)_{N-n_m} (q)_{n_{m}-n_{m-1}}\cdots (q)_{n_2-n_1}(q)_{n_1}(tq)_{n_1}}.
	\end{align}
\end{corollary}

In addition, it has been shown in \cite[Theorem~8.1]{shane2024multiple} that for $m\ge 2$,
\begin{align*}
	\cX_N^{(m)}(1,t,q) &= \frac{(t^2q)_\infty}{(tq)_\infty (q)_N} \sum_{n_1,\ldots,n_m\ge 0} (-1)^{n_m} t^{-n_m+\sum_{i=1}^m 2n_i} q^{-\binom{n_m}{2}+\sum_{i=1}^m n_i^2} \notag\\
	&\quad\times \frac{(t)_{n_m}}{(t^2q)_{N+n_1}(q)_{n_m}} \qbinom{n_m}{n_{m-1}}\qbinom{n_{m-1}}{n_{m-2}}\cdots \qbinom{n_2}{n_1}.
\end{align*}
Again, the $a$-independence for $\cX_N^{(m)}(a,t,q)$ gives us the following equality, which was shown earlier by the first author using a complicated method (see the proof of \cite[Theorem~8.2]{shane2024multiple}); now the $a$-independence sheds light on it in an alternative way.

\begin{corollary}
	For $m\ge 2$,
	\begin{align}
		&\sum_{n_1,\ldots,n_m\ge 0} \frac{(-1)^{n_m} t^{-n_m+\sum_{i=1}^m 2n_i} q^{-\binom{n_m}{2}+\sum_{i=1}^m n_i^2} (t)_{n_m}}{(t^2q)_{N+n_1}(q)_{n_m}} \qbinom{n_m}{n_{m-1}}\qbinom{n_{m-1}}{n_{m-2}}\cdots \qbinom{n_2}{n_1}\notag\\
		&\quad = \frac{(tq)_\infty (q)_N}{(t^2q)_\infty} \sum_{n_1,\ldots,n_m\ge 0} \frac{t^{\sum_{i=1}^m 2n_i} q^{\sum_{i=1}^m n_i^2}}{(q)_{N-n_m} (q)_{n_{m}-n_{m-1}}\cdots (q)_{n_2-n_1}(q)_{n_1}(tq)_{n_1}}.
	\end{align}
\end{corollary}

\section{Deformations, bona et mala}\label{sec:t-deform}

It remains to answer a \emph{conceptual} question --- \textit{How ``good'' a deformation for a multiple Rogers--Ramanujan type sum should be?}

For the Andrews--Gordon sum, the understanding is not at all controversial. Clearly, the most natural choice is
\begin{align}\label{eq:dAG-finite-t}
	\dAG_{n}^{(2m+3)} (t,q) := (q)_n \sum_{n_1,\ldots,n_m\ge 0} \frac{t^{\sum_{i=1}^m n_i} q^{\sum_{i=1}^m n_i^2}}{(q)_{n-n_m}(q)_{n_m-n_{m-1}}\cdots (q)_{n_2-n_1} (q)_{n_1}}.
\end{align}
However, this simply reduces to our $\AG_{n}^{(2m+3)}(t,q)$ by replacing $t$ with $t^2$. It is worth pointing out that $\dAG_{n}^{(2m+3)} (t,q)$ serves as one of the motivating examples for the theory of Bailey chain \cite[Section~3.5]{And1986}, and indeed we can simplify it as a single sum
\begin{align}\label{eq:dAG-simple}
	\dAG_{n}^{(2m+3)} (t,q) = (q)_n \sum_{r\ge 0} \frac{(-1)^r t^{(m+1)r} q^{\binom{r}{2}+(m+1)r^2} (1-tq^{2r}) (t)_r}{(1-t)(tq)_{n+r}(q)_{n-r}(q)_r},
\end{align}
by using a \emph{Bailey pair} relative to $(t,q)$ \cite[p.~31, eqs.~(3.47) and (3.48)]{And1986}:
\begin{align*}
	\alpha_k = \frac{(-1)^k q^{\binom{k}{2}} (1-tq^{2k}) (t)_k}{(1-t)(q)_k} \qquad\text{and}\qquad \beta_k = \begin{cases}
		1, & k=0,\\
		0, & k\ge 1.
	\end{cases}
\end{align*}
Such a way of simplification provides a coarse guide to the identification of good deformations on the $q$-hypergeometric side.

For the Bressoud sum, the story becomes bifurcating.

Again, from a $q$-theoretic perspective, it is the most natural to choose
\begin{align}\label{eq:dBr-finite-t}
	\dBr_n^{(2m+2)}(t,q) := (q)_n \sum_{n_1,\ldots,n_m\ge 0} \frac{t^{\sum_{i=1}^m n_i} q^{\sum_{i=1}^m n_i^2}}{(q)_{n-n_m}(q)_{n_m-n_{m-1}}\cdots (q)_{n_2-n_1} (q^2;q^2)_{n_1}}.
\end{align}
Then we have the reduction to the $m$-fold Bressoud sum by $\dBr_n^{(2m+2)}(1,q) = \Br_n^{(2m+2)}(1,q)$. Applying the following Bailey pair relative to $(t,q)$ \cite[p.~441, eq.~(D4)]{BIS2000}:
\begin{align*}
	\alpha_k = \frac{(-1)^k q^{k^2} (1-tq^{2k}) (t^2;q^2)_k}{(1-t)(q^2;q^2)_k} \qquad\text{and}\qquad \beta_k = \frac{1}{(q^2;q^2)_k},
\end{align*}
we have a single-sum simplification analogous to \eqref{eq:dAG-simple}:
\begin{align}\label{eq:dBr-simple}
	\dBr_n^{(2m+2)}(t,q) = (q)_n \sum_{r\ge 0} \frac{(-1)^r t^{mr} q^{(m+1)r^2} (1-tq^{2r}) (t^2;q^2)_r}{(1-t)(tq)_{n+r}(q)_{n-r}(q^2;q^2)_r}.
\end{align}
It is also notable that there is another slightly different $m$-fold bivariate sum, which does not reduce to the Bressoud sum $\Br_n^{(2m+2)}(1,q)$ but still has its own interest. Namely,
\begin{align}\label{eq:ddBr-finite-t}
	\ddBr_n^{(2m+2)}(t,q) := (q)_n \sum_{n_1,\ldots,n_m\ge 0} \frac{(-t)^{\sum_{i=1}^m n_i} q^{\sum_{i=1}^m n_i^2}}{(q)_{n-n_m}(q)_{n_m-n_{m-1}}\cdots (q)_{n_2-n_1} (q)_{n_1}(-tq)_{n_1}}.
\end{align}
It is true that
\begin{align}
	\ddBr_n^{(2m+2)}(t,q) = \frac{1}{(-tq)_n},
\end{align}
by invoking the trivial Bailey pair relative to $(-t,q)$:
\begin{align*}
	\alpha_k = \begin{cases}
		1, & k=0,\\
		0, & k\ge 1,
	\end{cases}\qquad\text{and}\qquad \beta_k = \frac{1}{(q)_k(-tq)_k}.
\end{align*}

On the other hand, unlike the case of Andrews--Gordon sum, in which $\AG_{n}^{(2m+3)} (t,q) = \dAG_{n}^{(2m+3)} (t^2,q)$, we no longer have a direct relation between $\Br_n^{(2m+2)}(t,q)$ and $\dBr_n^{(2m+2)}(t,q)$. Now what makes $\Br_n^{(2m+2)}(t,q)$ still a good $t$-deformation, especially noting that $\Br_n^{(2m+2)}(t,q)$ seems unable to be shortened as a single sum? The relations in \eqref{eq:coh-zeta-split} and \eqref{eq:main} give us an answer --- the \emph{geometric} meaning.

Nevertheless, it is undeniable that treatments of geometric objects are usually complicated, and often we are only led to some intermediate expressions such as that in \eqref{eq:zeta-new}. Hence, a meaningful question is how to translate the high-end geometric insights into the more practical $q$-theoretic setting.

A hint to this inquiry lies in the relation
\begin{align}
	\Br_n^{(2m+2)}(t,q) = \frac{t^{(2m-1)n} q^{\binom{n+1}{2}+(m-1)n^2} (1+t)}{1+t q^n} \Br_n^{(2m+2)}(t^{-1}q^{-n},q),
\end{align}
appearing in the earlier work of the first author \cite[Theorem~1.8]{shane2024multiple}. More importantly, such a $t\mapsto t^{-1}q^{-n}$ reflection formula is \emph{universal} in the geometric setting, as demonstrated by the second author and Jiang \cite[Theorem~7.1]{huangjiang2023torsionfree}. To be specific, letting $R$ be any arithmetic local order, and $\Omega_R$ its dualizing module, then with $\Delta:=|\tl R/R|$ the Serre invariant, the normalized Quot zeta function $\nu_{\Omega_R^n}^R(s)$ given in \eqref{eq:rational-formula-arithmetic} satisfies the \emph{reflection principle}\footnote{In particular, when $R=R_{2,2m+1}$, $R_{2,2m}$, or $R_{2,2m}'$, we have $\Omega_R=R$ and $\Delta=q^m$.}:
\begin{align}\label{eq:geo-reflection}
	\nu_{\Omega_R^n}^R(s) = \Delta^{n^2-2ns}\nu_{\Omega_R^n}^R(n-s);
\end{align}
the $t\mapsto t^{-1}q^{-n}$ rule comes from the fact that $t^{-1}q^{-n} = q^{-(n-s)}$ where we have always written $t:=q^{-s}$. As another instance, this reflective phenomenon also exists for our $t$-deformed Andrews--Gordon sum $\AG_n^{(2m+3)}(t,q)$. Replacing $t$ with $t^{-1}q^{-n}$ in \eqref{eq:AG-finite-t} and then making the substitution $n_i \mapsto n-n_{m-i+1}$ for each $i$ with $1\le i\le m$, we immediately have
\begin{align}\label{eq:AG-reflection}
	\AG_n^{(2m+3)}(t,q) = t^{2mn} q^{mn^2} \AG_n^{(2m+3)}(t^{-1}q^{-n},q).
\end{align}
In the meantime, the relation $\AG_{n}^{(2m+3)} (t,q) = \dAG_{n}^{(2m+3)} (t^2,q)$ implies that
\begin{align}
	\dAG_n^{(2m+3)}(t,q) = t^{mn} q^{mn^2} \dAG_n^{(2m+3)}(t^{-1}q^{-2n},q).
\end{align}
The agreement with the $t\mapsto t^{-1}q^{-n}$ rule in \eqref{eq:AG-reflection} makes it clear why we have chosen $\AG_n^{(2m+3)}(t,q)$ as our $t$-deformation, which involves $t^2$ instead of the more natural $t$ in the numerator of the summand.

All in all, this reflective nature may serve as another criterion for good deformations for multiple Rogers--Ramanujan type sums. In addition, once such a deformation is constructed, it is reasonable to expect it to convey a certain geometric meaning.

\section{A master interpolation}\label{sec:further-interpolations}

\subsection{A master polynomial}

The role the $q,t$-series $\nu^R_{\tl R^n}(s)$ (where $t:=q^{-s}$) plays in the geometric framework has remained a mystery. For one thing, $\nu^R_{\tl R^n}(0)=\nu^R_{R^n}(0)$ by \cite[Corollary~5.12]{huangjiang2023torsionfree}, so $\nu^R_{\tl R^n}(s)$ and $\nu^R_{R^n}(s)$ are \emph{different} $t$-deformations of the same $q$-series. In some sense, $\nu^R_{\tl R^n}(s)$ can be treated as a ``surrogate'' for the ``real'' deformation $\nu^R_{R^n}(s)$ --- the former is easier to compute and enough to give the $t=1$ specialization, but is ``worse'' due to the lack of a universal reflection principle as in \eqref{eq:geo-reflection}.

For rank $n=1$, these deformations are predicted to be ``essentially the same'' by the \emph{Hilb-vs-Quot conjecture} \cite{kivinentrinh23hilb}:
\begin{equation}\label{eq:tl-1}
	\nu^R_{R}(s) = \nu^R_{\tl R}(s) \big|_{q\mapsto tq},
\end{equation}
but such a simple relation no longer holds when $n>1$ and $R=R_{2,2m}$ as shown in \cite{huangjiang2023torsionfree}. To make things more mysterious, a \emph{different} relation holds for $R_{2,2m+1}$:
\begin{equation}\label{eq:tl-n-2m+1}
	\nu^{R_{2,2m+1}}_{R_{2,2m+1}^n}(s) = \nu^{R_{2,2m+1}}_{\tl R_{2,2m+1}^n}(s) \big|_{t\mapsto t^2},
\end{equation}
which happens to coincide with the rank-$1$ rule because $\nu^{R_{2,2m+1}}_{\tl R_{2,2m+1}^n}(s)\in \Z[tq]$. Another indication of why $\nu^R_{\tl R^n}(s)$ appears to be a more complicated object is that for $R=R_{2,2m}$ and $R_{2,2m}'$, a simple $t\mapsto -t$ relation connects these two $\nu^R_{R^n}(s)$, while the two $\nu^R_{\tl R^n}(s)$ do not seem to relate one another directly. These observations have led to a general guess that there is possibly a hidden extra variable, with some unclear geometric origin, that interpolates $\nu^R_{\tl R^n}(s)$ and $\nu^R_{R^n}(s)$; whatever it is, it must degenerate when $n=1$, and degenerate in a different way when $R=R_{2,2m+1}$.

While we are not yet able to give further geometric insights to the hidden variable, we observe certain apparent interpolations from a purely $q$-theoretic point of view. % For below, let $R=R_{2,2m}$ and $R'=R_{2,2m}'$.
To be specific, we formulate the relevant Quot zeta functions in a way that seems to reveal the parallelism the most.

\begin{proposition}
	Let $R=R_{2,2m}$ and $R'=R_{2,2m}'$. Then with $z:=q^{-1}$ and $t:=q^{-s}$,
	\begin{align*}
		\nu^R_{\tl R^n}(s)&=\sum_{n_1,\dots,n_m\ge 0}  \frac{t^{mn-\sum_{i=1}^m n_i} z^{-mn^2+\sum_{i=1}^m n_i^2} (z;z)_n^2}{(z;z)_{n-n_m}(z;z)_{n_{m}-n_{m-1}}\cdots (z;z)_{n_{2}-n_1}(z;z)_{n_1}^2},\\
		\nu^{R'}_{\tl{R'}^n}(s)&=\sum_{n_1,\dots,n_m\ge 0}  \frac{t^{mn-\sum_{i=1}^m n_i} z^{-mn^2+\sum_{i=1}^m n_i^2} (z;z)_n (-z;z)_n}{(z;z)_{n-n_m}(z;z)_{n_{m}-n_{m-1}}\cdots (z;z)_{n_{2}-n_1}(z;z)_{n_1}(-z;z)_{n_1}},\\
		\nu^R_{R^n}(s)&=\sum_{n_1,\dots,n_m\ge 0}  \frac{t^{mn-\sum_{i=1}^m n_i} z^{-mn^2+\sum_{i=1}^m n_i^2} (z;z)_n (t^{-1}z;z)_n}{(z;z)_{n-n_m}(z;z)_{n_{m}-n_{m-1}}\cdots (z;z)_{n_{2}-n_1}(z;z)_{n_1}(t^{-1}z;z)_{n_1}},\\
		\nu^{R'}_{R'^n}(s)&=\sum_{n_1,\dots,n_m\ge 0}  \frac{t^{mn-\sum_{i=1}^m n_i} z^{-mn^2+\sum_{i=1}^m n_i^2} (z;z)_n (-t^{-1}z;z)_n}{(z;z)_{n-n_m}(z;z)_{n_{m}-n_{m-1}}\cdots (z;z)_{n_{2}-n_1}(z;z)_{n_1}(-t^{-1}z;z)_{n_1}}.
	\end{align*}
\end{proposition}

\begin{proof}
	The relation for $\nu^R_{\tl R^n}(s)$ follows from \cite[Proposition~9.6]{huangjiang2023torsionfree}; for $\nu^{R'}_{\tl{R'}^n}(s)$, we recall Proposition~\ref{prop:rtilde}; for $\nu^R_{R^n}(s)$, we use \eqref{eq:coh-zeta-split} and the reflection principle \eqref{eq:geo-reflection}; and finally, for $\nu^{R'}_{R'^n}(s)$, we use Theorem~\ref{thm:main} again combining with the reflection principle.
\end{proof}

As such, we propose the following \emph{conceptual} interpolation that incorporates the three quadratic orders. Define a \emph{master polynomial} in $\Z[u,t,z^{-1}]$:
\begin{align}
	&\tl \cyZh_{m,n}(u,t,z)\notag\\
	&\quad:=\sum_{n_1,\dots,n_m\ge 0} \frac{u^{mn-\sum_{i=1}^m n_i}z^{-mn^2+\sum_{i=1}^m n_i^2}(z;z)_n(u^{-1}tz;z)_n}{(z;z)_{n-n_m}(z;z)_{n_m-n_{m-1}}\cdots (z;z)_{n_{2}-n_1}(z;z)_{n_1}(u^{-1}tz;z)_{n_1}}.
\end{align}

\begin{remark}
	This master polynomial parametrizes all our deformations for Andrews--Gordon and Bressoud sums investigated earlier by the relations:
	\begin{align*}
		\AG_n^{(2m+3)}(t,q) &= t^{2mn}q^{mn^2} \tl \cyZh_{m,n}(t^{-2},0,q),\\
		\dAG_n^{(2m+3)}(t,q) &= t^{mn}q^{mn^2} \tl \cyZh_{m,n}(t^{-1},0,q),\\
		\Br_n^{(2m+2)}(t,q) &= \frac{t^{2mn}q^{mn^2}}{(-tq;q)_n} \tl \cyZh_{m,n}(t^{-2},-t^{-1},q),\\
		\dBr_n^{(2m+2)}(t,q) &= \frac{t^{mn}q^{mn^2}}{(-q;q)_n} \tl \cyZh_{m,n}(t^{-1},-t^{-1},q),\\
		\ddBr_n^{(2m+2)}(t,q) &= \frac{(-t)^{mn}q^{mn^2}}{(-tq;q)_n} \tl \cyZh_{m,n}(-t^{-1},1,q).
	\end{align*}
\end{remark}

Now define
\begin{align}
	\cyZh_{R,n}(u,t,q):=\tl \cyZh_{m,n}(u,\epsilon_R\, t,q^{-1}),
\end{align}
where
\begin{align*}
	\epsilon_R := \begin{cases}
		0, & \text{if $R=R_{2,2m+1}$},\\
		1, & \text{if $R=R_{2,2m}$},\\
		-1, & \text{if $R=R_{2,2m}'$}.
	\end{cases}
\end{align*}
With this, we may unify the normalized Quot zeta functions $\nu^R_{\tl R^n}(s)$ and $\nu^R_{R^n}(s)$ for all three quadratic orders.

\begin{corollary}
	For $R=R_{2,2m+1}$, $R_{2,2m}$, and $R_{2,2m}'$, we have
	\begin{align}
		\nu^R_{\tl R^n}(s)&=\cyZh_{R,n}(t,t,q),\label{eq:interpolate-1}\\
		\nu^R_{R^n}(s)&=\cyZh_{R,n}(t^2,t,q).\label{eq:interpolate-2}
	\end{align}
\end{corollary}

\begin{remark}
	The $t$-independence of $\cyZh_{R_{2,2m+1},n}(u,t,q)$ directly explains the $t\mapsto t^2$ rule in \eqref{eq:tl-n-2m+1}. 
\end{remark}

\begin{remark}
	In the rank $n=1$ case, we notice that
	\begin{align*}
		\cyZh_{R,1}(u,t,q)=1+\big(1-\epsilon_R\,  t(uq)^{-1}\big)\sum_{r=1}^m (uq)^r.
	\end{align*}
	For all three quadratic orders, it lives in $\Z[uq,t]$, thereby recovering the $q\mapsto tq$ rule in \eqref{eq:tl-1}.
\end{remark}

\begin{remark}
	The $\epsilon_R$ dictionary is \emph{not} an unexpected feature, because it echoes the Legendre symbol and appears in the rank-$1$ story \cite{saikia88}.
\end{remark}

\subsection{Reflection principle}

To demonstrate why our master polynomial may encode some yet-to-be-determined geometric meaning, we establish the following reflective functional equation, resonating with the discussion in Section~\ref{sec:t-deform}.

\begin{theorem}\label{thm:master-reflection}
	The master polynomial satisfies
	\begin{align}\label{eq:master-reflection}
		\tl \cyZh_{m,n}(u,t,z)=u^{mn}z^{-mn^2}\tl\cyZh_{m,n}(u^{-1}z^{2n},u^{-1}tz^n,z).
	\end{align}
	In particular, for $R=R_{2,2m+1}$, $R_{2,2m}$, and $R_{2,2m}'$, we have
	\begin{align}\label{eq:fine-reflection}
		\cyZh_{R,n}(u,t,q)=u^{\delta_R n}q^{\delta_R n^2}\cyZh_{R,n}(u^{-1}z^{2n},u^{-1}tz^n,q),
	\end{align}
	where $\delta_R=\dim_{\Fq} \tl R/R$.
\end{theorem}

\begin{proof}
	Using the relation
	\begin{align*}
		\frac{(tz^{1-n};z)_n}{(tz^{1-n};z)_{n_1}} = (-t)^{n-n_1} z^{-\binom{n-n_1}{2}} (t^{-1};z)_{n-n_1},
	\end{align*}
	and then making the change of indices $n_i\mapsto n-n_{m-i+1}$ for each $i$ with $1\le i\le m$, we have
	\begin{align*}
		&\cyZh_{m,n}(u^{-1}z^{2n},u^{-1}tz^n,z)\\
		&\quad = \sum_{n_1,\dots,n_m\ge 0} \frac{(-t)^{n_m} u^{-\sum_{i=1}^m n_i}z^{-\binom{n_m}{2}+\sum_{i=1}^m n_i^2}(z;z)_n(t^{-1};z)_{n_m}}{(z;z)_{n-n_m}(z;z)_{n_m-n_{m-1}}\cdots (z;z)_{n_{2}-n_1}(z;z)_{n_1}}.
	\end{align*}
	Hence, \eqref{eq:master-reflection} is equivalent to
	\begin{align}\label{eq:master-reflection-new}
		&\sum_{n_1,\dots,n_m\ge 0} \frac{(-t)^{n_m} u^{-\sum_{i=1}^m n_i}z^{-\binom{n_m}{2}+\sum_{i=1}^m n_i^2}(t^{-1};z)_{n_m}}{(z;z)_{n-n_m}(z;z)_{n_m-n_{m-1}}\cdots (z;z)_{n_{2}-n_1}(z;z)_{n_1}}\notag\\
		&\quad = \sum_{n_1,\dots,n_m\ge 0} \frac{u^{-\sum_{i=1}^m n_i}z^{\sum_{i=1}^m n_i^2}(u^{-1}tz;z)_n}{(z;z)_{n-n_m}(z;z)_{n_m-n_{m-1}}\cdots (z;z)_{n_{2}-n_1}(z;z)_{n_1}(u^{-1}tz;z)_{n_1}}.
	\end{align}
	For the left-hand side of \eqref{eq:master-reflection-new}, we single out the sum over $n_m$,
	\begin{align*}
		\LHS\eqref{eq:master-reflection-new} &= \sum_{n_1,\dots,n_{m-1}\ge 0} \frac{u^{-\sum_{i=1}^{m-1} n_i}z^{\sum_{i=1}^{m-1} n_i^2}}{(z;z)_{n_{m-1}-n_{m-2}}\cdots (z;z)_{n_{2}-n_1}(z;z)_{n_1}}\\
		&\quad\times \sum_{n_m\ge n_{m-1}} \frac{(-u^{-1}tz)^{n_m} z^{\binom{n_m}{2}} (t^{-1};z)_{n_m}}{(z;z)_{n-n_m} (z;z)_{n_m-n_{m-1}}}.
	\end{align*}
	Now we recall \cite[Lemma~8.3 with $a=0$]{shane2024multiple}:
	\begin{align}\label{eq:Z-1=Z-2-full}
		\sum_{n\ge L} \frac{(-tb^{-1}q)^{n} q^{\binom{n}{2}} (b)_{n}}{(q)_{M-n} (q)_{n-L}} =(-b^{-1})^L q^{-\binom{L}{2}}(b)_L(b^{-1}tq)_M \sum_{n\ge L} \frac{t^{n}q^{n^2}}{(q)_{M-n}(q)_{n-L} (b^{-1}tq)_{n}},
	\end{align}
	which holds for nonnegative integers $L$ and $M$. Applying \eqref{eq:Z-1=Z-2-full} with $(b,t,q)\mapsto (t^{-1},u^{-1},z)$ and $(L,M)= (n_{m-1},n)$, we have
	\begin{align*}
		&\LHS\eqref{eq:master-reflection-new} \\
		&\quad= \sum_{n_m\ge 0} \frac{u^{-n_m} z^{n_m^2} (u^{-1}tz;z)_n}{(z;z)_{n_m} (u^{-1}tz;z)_{n_m}} \sum_{n_1,\dots,n_{m-2}\ge 0} \frac{u^{-\sum_{i=1}^{m-2} n_i}z^{\sum_{i=1}^{m-2} n_i^2}}{(z;z)_{n_{m-2}-n_{m-3}}\cdots (z;z)_{n_{2}-n_1}(z;z)_{n_1}}\\
		&\quad\quad \times \sum_{n_{m-1}\ge n_{m-2}} \frac{(-u^{-1}tz)^{n_{m-1}} z^{\binom{n_{m-1}}{2}} (t^{-1};z)_{n_{m-1}}}{(z;z)_{n_m-n_{m-1}} (z;z)_{n_{m-1}-n_{m-2}}}.
	\end{align*}
	For the sum over $n_{m-1}$, we may again apply \eqref{eq:Z-1=Z-2-full} with $(b,t,q)\mapsto (t^{-1},u^{-1},z)$ and $(L,M)= (n_{m-2},n_m)$. Iterating this process, we eventually arrive at the right-hand side of \eqref{eq:master-reflection-new}, thereby establishing \eqref{eq:master-reflection}.
\end{proof}

We are yet to find the geometric meaning of $\cyZh_{R,n}(u,t,q)$, but the \emph{uniform} relations \eqref{eq:interpolate-1}, \eqref{eq:interpolate-2}, and \eqref{eq:fine-reflection} suggest that there might be an analogous trivariate polynomial $\cyZh_{R,n}(u,t,q)$ defined for other planar curve germs $R$. Once such a generic $\cyZh_{R,n}(u,t,q)$ is constructed, if one further shows that the rank-$1$ case $\cyZh_{R,1}(u,t,q)$ lives in $\Z[uq,t]$, then this would lead us to a proof of the Hilb-vs-Quot conjecture.

\subsection{Cyclic sieving}
In \cite{huangjiang2023torsionfree}, a curious ``\emph{cyclic sieving}'' phenomenon is observed for $\nu^R_{R^n}(s)$ and $\nu^R_{\tl R^n}(s)$, concerning the evaluation for $q$ at roots of unity. Namely, if $r$ divides $n$ and $\zeta_r$ is a primitive $r$-th root of unity, then
\begin{equation}\label{eq:cyclic-sieving}
	\nu^R_{E^n}(s)|_{q\mapsto \zeta_r}\overset{?}{=}\parens*{\nu^R_{E}(rs)|_{q\mapsto 1}}^{n/r},
\end{equation}
for $E=R$ or $\tl R$.
The $q=1$ substitution (case $r=1$) is fully explained by geometry \cite[Proposition~8.13]{huangjiang2023torsionfree}, but the $r>1$ cases remain a mystery. This phenomenon gives another side to the question ---
\begin{quote}
	``\textit{Does the rank-$n$ Donaldson--Thomas theory factor through rank-$1$?}''
\end{quote}
While a positive answer has been witnessed in various settings on Calabi--Yau $3$-folds \cite{FMR2021,FT2023}, prior calculation in \cite{huangjiang2023torsionfree} shows that it is not the case at least for the motivic degree zero theory on the singular curve $R_{2,M}$. However, the cyclic sieving phenomenon says at least it is so at $n$-th roots of unity (primitive or not). Here we show that this phenomenon holds for the master polynomial $\tl \cyZh_{m,n}(u,t,z)$, which recovers the prediction \eqref{eq:cyclic-sieving} for all three types of quadratic orders $R=R_{2,2m+1}$, $R_{2,2m}$, and $R_{2,2m}'$.

\begin{theorem}\label{thm:Zh-at-root}
	For $r\mid n$, letting $\zeta_r$ be a primitive $r$-th root of unity, then
	\begin{align}\label{eq:Zh-at-root}
		\tl \cyZh_{m,n}(u,t,\zeta_r) = \left(\frac{1-t^r+u^{mr}t^r-u^{(m+1)r}}{1-u^r}\right)^{\frac{n}{r}}.
	\end{align}
	In particular,
	\begin{align}
		\tl \cyZh_{m,n}(u,t,\zeta_r) = \tl \cyZh_{m,1}(u^r,t^r,1)^{n/r}.
	\end{align}
\end{theorem}

\begin{proof}
	In terms of the \emph{$q$-multinomial coefficients}
	\begin{align*}
		\qbinom{N}{N_1,\ldots,N_s}_q := \frac{(q;q)_N}{(q;q)_{N_1}\cdots (q;q)_{N_s}},
	\end{align*}
	defined for nonnegative integers $N_1,\ldots,N_s$ and $N$ with $N=N_1+\cdots+N_s$, we may rewrite $\tl \cyZh_{m,n}(u,t,z)$ as
	\begin{align*}
		&\tl \cyZh_{m,n}(u,t,z)\\
		&\quad = \sum_{\substack{k_0,\ldots,k_m\ge 0\\k_0+\cdots+k_m=n}} \frac{u^{\sum_{i=1}^m i k_i} z^{-mn^2+\sum_{i=1}^m (n-k_i-\cdots-k_m)^2}(u^{-1}tz;z)_n}{(u^{-1}tz;z)_{k_0}}\qbinom{n}{k_0,\ldots,k_m}_z,
	\end{align*}
	where we have applied the substitution $k_i := n_{i+1}-n_i$ for each $i$ with $0\le i\le m$ under the convention that $n_{m+1}:=n$ and $n_0:=0$. An obvious fact about the $q$-multinomial coefficients is that assuming $r\mid N$, we always have, for $\zeta_r$ a primitive $r$-th root of unity,
	\begin{align*}
		\qbinom{N}{N_1,\ldots,N_s}_{\zeta_r} = \begin{cases}
			\binom{N/r}{N_1/r,\ldots,N_s/r}, & \text{if $r\mid N_i$ for every $i$},\\[6pt]
			0, & \text{otherwise},
		\end{cases}
	\end{align*}
	where
	\begin{align*}
		\binom{N}{N_1,\ldots,N_s} := \frac{N!}{N_1!\cdots N_s!}
	\end{align*}
	are the usual \emph{multinomial coefficients} defined again for nonnegative integers $N_1,\ldots,N_s$ and $N$ with $N=N_1+\cdots+N_s$. Now,
	\begin{align*}
		\tl \cyZh_{m,n}(u,t,\zeta_r) = \sum_{\substack{k'_0,\ldots,k'_m\ge 0\\k'_0+\cdots+k'_m=\frac{n}{r}}} (u^r)^{\sum_{i=1}^m i k'_i} (u^{-1}t\zeta_r;\zeta_r)_r^{\frac{n}{r}-k'_0} \binom{\frac{n}{r}}{k'_0,\ldots,k'_m}.
	\end{align*}
	Noting that
	\begin{align*}
		(u^{-1}t\zeta_r;\zeta_r)_r = 1-u^{-r}t^r,
	\end{align*}
	we further have
	\begin{align*}
		\tl \cyZh_{m,n}(u,t,\zeta_r) &= (1-u^{-r}t^r)^{\frac{n}{r}} \sum_{\substack{k'_0,\ldots,k'_m\ge 0\\k'_0+\cdots+k'_m=\frac{n}{r}}} \binom{\frac{n}{r}}{k'_0,\ldots,k'_m} \frac{(u^r)^{\sum_{i=1}^m i k'_i}}{(1-u^{-r}t^r)^{k'_0}}\\
		&= (1-u^{-r}t^r)^{\frac{n}{r}} \left(\frac{1}{1-u^{-r}t^r}+u^r+u^{2r}+\cdots+u^{mr}\right)^{\frac{n}{r}},
	\end{align*}
	where we have used the multinomial theorem for the sum. The claimed relation \eqref{eq:Zh-at-root} therefore follows.
\end{proof}

We should point out that cyclic sieving is not just about good evaluations at $n$-th roots of unity; instead, a good evaluation is just one manifestation of the full combinatorics of the cyclic sieving phenomenon, which involves a finite set $X$ with an action by the cyclic group $\Z/n\Z$, as first explored in \cite{reinerstantonwhite}.

The observation in Theorem~\ref{thm:Zh-at-root}, therefore, strongly suggests that a \emph{combinatorial} model --- such as a weighted count on a finite set --- might exist for $\tl \cyZh_{m,n}(u,t,z)$. The search for such a model is often complementary to the search for an underlying \emph{geometric} model (e.g., a moduli space), which has been proposed in the discussion after Theorem~\ref{thm:master-reflection}. Frequently, the combinatorial model is tightly connected with the set of torus-fixed points of the geometric one. The classical relationship between binomial coefficients and $q$-binomial coefficients illustrates this symbiosis perfectly:
$$\binom{n}{k}=\left\lvert\binom{[n]}{k}\right\rvert\qquad\text{versus}\qquad \qbinom{n}{k}_q=\abs{\Gr(k,n)(\Fq)}.$$
Here, the geometric model is the Grassmannian variety $\Gr(k,n)$. Its point count over $\Fq$ yields the $q$-analog, while its set of torus-fixed points is in canonical bijection with the combinatorial model $\binom{[n]}{k}$, the set of $k$-element subsets of $[n]:=\set{1,\dots,n}$.

\subsection*{Acknowledgements}

Shane Chern was supported by the Austrian Science Fund (No.~10.55776/F1002). Yifeng Huang thanks Ruofan Jiang for helpful conversations.

\bibliographystyle{plain}

\begin{thebibliography}{[XXX99]}
	
	\bibitem[And74]{And1974}
	G. E. Andrews, An analytic generalization of the Rogers--Ramanujan identities for odd moduli, \textit{Proc. Nat. Acad. Sci. U.S.A.} \textbf{71} (1974), 4082--4085.
	
	\bibitem[And86]{And1986}
	G. E. Andrews, \textit{$q$-Series: their development and application in analysis, number theory, combinatorics, physics, and computer algebra}, American Mathematical Society, Providence, RI, 1986.
	
	\bibitem[And98]{And1998}
	G. E. Andrews, \textit{The theory of partitions}, Cambridge University Press, Cambridge, 1998.
	
	\bibitem[Bre80]{Bre80}
	D. M. Bressoud, An analytic generalization of the Rogers--Ramanujan identities with interpretation, \textit{Quart. J. Math. Oxford Ser. (2)} \textbf{31} (1980), no. 124, 385--399.
	
	\bibitem[BIS00]{BIS2000}
	D. Bressoud, M. E. H. Ismail, and D. Stanton, Change of base in Bailey pairs, \textit{Ramanujan J.} \textbf{4} (2000), no. 4, 435--453.
	
	\bibitem[Che24]{shane2024multiple}
	S. Chern, Multiple Rogers--Ramanujan type identities for torus links, preprint, 2024. Available at arXiv:2411.07198.
	
	\bibitem[FMR21]{FMR2021}
	N. Fasola, S. Monavari, and A. T. Ricolfi, Higher rank K-theoretic Donaldson--Thomas theory of points, \textit{Forum Math. Sigma} \textbf{9} (2021), Paper No. e15, 51 pp.
	
	\bibitem[FT23]{FT2023}
	S. Feyzbakhsh and R. P. Thomas, Rank $r$ DT theory from rank $1$, \textit{J. Amer. Math. Soc.} \textbf{36} (2023), no. 3, 795--826.
	
	\bibitem[GR04]{GR2004}
	G. Gasper and M. Rahman, \textit{Basic hypergeometric series. Second edition}, Cambridge University Press, Cambridge, 2004.
	
	\bibitem[Gor61]{Gor1961}
	B. Gordon, A combinatorial generalization of the Rogers--Ramanujan identities, \textit{Amer. J. Math.} \textbf{83} (1961), 393--399.
	
	\bibitem[GM13]{gorskymazin2013compactified1}
	E. Gorsky and M. Mazin, Compactified Jacobians and {$q,t$}-Catalan numbers, I, \textit{J. Combin. Theory Ser. A} \textbf{120} (2013), no. 1, 49--63.
	
	\bibitem[GMO23]{gmo2023generic}
	E. Gorsky, M. Mazin, and A. Oblomkov, Generic curves and non-coprime Catalans, \textit{S\'em. Lothar. Combin.} \textbf{89B} (2023), Art. 63, 12 pp.
	
	\bibitem[GMV16]{gmv2016affine}
	E. Gorsky, M. Mazin, and M. Vazirani, Affine permutations and rational slope parking functions, \textit{Trans. Amer. Math. Soc.} \textbf{368} (2016), no. 12, 8403--8445.
	
	\bibitem[GMV17]{gmv2017rational}
	E. Gorsky, M. Mazin, and M. Vazirani, Rational Dyck paths in the non relatively prime case, \textit{Electron. J. Combin.} \textbf{24} (2017), no. 3, Paper No. 3.61, 29 pp.
	
	\bibitem[Hua25]{huang2025coh}
	Y. Huang, Coh zeta functions for inert quadratic orders, preprint, 2025. Available at arXiv:2507.21966.
	
	\bibitem[HJ23]{huangjiang2023torsionfree}
	Y. Huang and R. Jiang, Motivic Coh and Quot zeta functions of singular curves, preprint, 2023. Available at arXiv:2312.12528.
	
	\bibitem[KT25]{kivinentrinh23hilb}
	O. Kivinen and M.-T. Trinh, The Hilb-vs-Quot conjecture, \textit{J. Reine Angew. Math.} \textbf{828} (2025), 83--126.
	
	\bibitem[Mac15]{macdonaldsymmetric}
	I. G. Macdonald, \textit{Symmetric functions and Hall polynomials. Second edition}, Oxford University Press, New York, 2015.
	
	\bibitem[ORS18]{ors2018homfly}
	A. Oblomkov, J. Rasmussen, and V. Shende, The Hilbert scheme of a plane curve singularity and the HOMFLY homology of its link, \textit{Geom. Topol.} \textbf{22} (2018), no. 2, 645--691.
	
	\bibitem[Ram14]{Ram1914}
	S. Ramanujan, Problem 584, \textit{J. Indian Math. Soc.} \textbf{6} (1914), 199--200.
	
	\bibitem[RSW04]{reinerstantonwhite}
	V. Reiner, D. Stanton, and D. White, The cyclic sieving phenomenon, \textit{J. Combin. Theory Ser. A} 108 (2004), no. 1, 17--50.
	
	\bibitem[Ric17]{ricolfi17}
	A. T. Ricolfi, Local Donaldson--Thomas invariants and their refinements, Thesis (Ph.D.)--University of Stavanger. 2017. Available at \url{https://aricolfi.github.io/MATHS/phd-thesis.pdf}.
	
	\bibitem[Rog94]{Rog1894}
	L. J. Rogers, Second memoir on the expansion of certain infinite products, \textit{Proc. Lond. Math. Soc.} \textbf{25} (1893/94), 318--343.
	
	\bibitem[Sai88]{saikia88}
	P. K. Saikia, Zeta functions of orders in quadratic fields, \textit{Proc. Indian Acad. Sci. Math. Sci.} \textbf{98} (1988), no. 1, 31--42.
	
	\bibitem[Sol77]{solomon1977zeta}
	L. Solomon, Zeta functions and integral representation theory, \textit{Advances in Math.} \textbf{26} (1977), no. 3, 306--326.
	
	\bibitem[Tai70]{taitslin70}
	M. A. Taitslin, Elementary theories of lattices of subgroups, \textit{Algebra i Logika} \textbf{9} (1970), 473--483.
	
	\bibitem[War13]{warnaar2013remarks}
	S. O. Warnaar, Remarks on the paper ``Skew Pieri rules for Hall--Littlewood functions'' by Konvalinka and Lauve, \textit{J. Algebraic Combin.} \textbf{38} (2013), no. 3, 519--526.
	
\end{thebibliography}

\end{document}